\documentclass[11pt]{amsart}
\usepackage[all]{xy}
\usepackage{latexsym}
\usepackage{amssymb}
\usepackage{amsfonts}
\usepackage{amscd}
\usepackage{amsmath,amsthm}
\usepackage{mathrsfs}

\newtheorem{lemma}{Lemma}[section]
\newtheorem{proposition}[lemma]{Proposition}

\newtheorem{corollary}[lemma]{Corollary}

\newtheorem{Def}[lemma]{Definition}
\newtheorem{Lem}[lemma]{Lemma}
\newtheorem{Prop}[lemma]{Proposition}
\newtheorem{Thm}[lemma]{Theorem}
\newtheorem{Cor}[lemma]{Corollary}
{

}
\theoremstyle{definition}

\newtheorem{example}[lemma]{Example}

\theoremstyle{remark}

\numberwithin{equation}{section}

\newtheorem*{Rmk}{Remark}
\def\bmat{\begin{pmatrix}}
\def\emat{\end{pmatrix}}
\def\-{\smallsetminus}
\def\i{\iota}
\def\~{\widetilde}
\def\ol{\overline}

\def\phi{\varphi}

\def\deg{\text{deg }}

\def\<{\langle}
\def\>{\rangle}

\def\Gr{\operatorname {Gr}}
\def\gr{\operatorname {gr}}

\def\GrRep{\operatorname {GrRep}}

\def\Fin{\operatorname {Fin}}
\def\Tors{\operatorname {Tors}}
\def\tors{\operatorname {tors}}
\def\QGr{\operatorname {QGr}}
\def\qgr{\operatorname {qgr}}

\def\End{\operatorname {End}}

\def\Hom{\operatorname {Hom}}

\def\max{\operatorname {max}}

\def\Proj{\operatorname {Proj}}

\def\Mod{\operatorname {Mod}}

\def\Ext{\operatorname {Ext}}

\def\End{\operatorname {End}}

\def\GKdim{\operatorname {GKdim}}
\def\K0{\operatorname {K_0}}

\def\supp{\operatorname {supp}}
\def\11{\text{\bf 1}}

\def\NN{{\mathbb N}}

\def\PP{{\mathbb P}}
\def\QQ{{\mathbb Q}}

\def\ZZ{{\mathbb Z}}
\def\cal{\mathcal}

\def\cC{{\cal C}}

\def\cO{{\cal O}}
\def\cP{{\cal P}}

\def\sM{{\mathscr M}}

\def\sP{{\mathscr P}}

\def\mC{{\mathfrak C}}
\def\mD{{\mathfrak D}}

\def\bv{{\bf v}}

\def\bp{{\bf p}}
\title[Path algebras of finite GK-dimension.]{Path algebras and monomial algebras of finite GK-dimension as noncommutative homogeneous coordinate rings.}

\begin{document}

\author{Cody Holdaway}

\address{Department of Mathematics, Box 354350, Univ.
Washington, Seattle, WA 98195}

\email{codyh3@math.washington.edu}



\begin{abstract}
This article sets out to understand the categories $\QGr A$ where $A$ is either a monomial algebra or a path algebra of finite Gelfand-Kirillov dimension. The principle questions are: 1) What is the structure of the point modules up to isomorphism in $\QGr A$? 2) When is $\QGr A \equiv \QGr A'$? These two questions turn out to be intimately related.

It is shown that up to isomorphism in $\QGr A$, there are only finitely many point modules and these give all the simple objects in the category. Then, a finite quiver $E_A$, which can be constructed from the algebra $A$ rather simply, is associated to the category $\QGr A$. It is shown that the vertices of $E_A$ are in bijection with the point modules and the arrows are determined by the extensions between point modules. Lastly, it is shown that $\QGr A\equiv \QGr A'$ if and only if $E_A=E_{A'}$.
\end{abstract}

\maketitle

\pagenumbering{arabic}

\setcounter{section}{0}

\section{Introduction.}

Let $k$ be a field, the term algebra will mean $k$-algebra.

One of the basic methods for understanding noncommutative graded algebras geometrically is to understand the structure of their point modules. This idea was used effectively in \cite{ATV1} to better understand the three dimensional AS-regular algebras originally studied in \cite{AS}. 

Given an graded $k$-algebra $A=\oplus_{i\geq 0}A_i$, a point module over $A$ is a cyclic graded right $A$-module $M=\oplus_{i\geq 0} M_i$ such that $\dim_kM_i=1$ for all $i\geq 0$. Point modules are so named as they are thought of as the ``closed'' points of the noncommutative projective scheme determined by $A$. Let $\Gr A$ denote the category of $\ZZ$-graded right $A$-modules with degree preserving homomorphisms and $\QGr A$ the quotient of $\Gr A$ by the Serre subcategory $\Tors A$ consisting of all the torsion modules. Point modules determine simple objects in $\QGr A$, however, there may be other simple objects not coming from point modules.

Let $S(a,b,c)$ denote the $3$-dimensional Sklyanin algebra with parameters $(a,b,c)$ (assume $k=\ol{k}$). If
$$
(a,b,c)\notin \mD:=\{(0,0,1), (0,1,0), (1,0,0)\} \sqcup \{(a,b,c)\;|\; a^3=b^3=c^3\},
$$ 
then $S(a,b,c)$ is a Noetherian domain with finite GK-dimension, and the point modules are parameterized nicely by an elliptic curve with an automorphism. However, Smith shows in \cite{Sm3} that when $S(a,b,c)$ is a degenerate Sklyanin algebra, that is, if $(a,b,c) \in \mD$, then $S(a,b,c)$ is a monomial algebra with $3$ generators $u,v,w$ such that either $u^2=v^2=w^2=0$ or $uv=vw=wu=0$. In both cases, $S(a,b,c)$ is not Noetherian and has infinite GK-dimension. For the degenerate Sklyanin algebras, the point modules are no longer parameterized by an elliptic curve, in fact, they don't seem to be nicely parameterized at all.       

As another example, consider the quantum planes 
$$
A_q:=k\<x,y\>/(xy-qyx),
$$ 
where $q\in k$. If $q\neq 0$, then $A_q$ is a Noetherian domain with finite GK-dimension and the point modules are parameterized by $\PP^1$. When $q\neq 0$, point modules are isomorphic in $\QGr A_q$ if and only if they are isomorphic in $\Gr A_q$. Thus, the point modules are still parameterized by $\PP^1$ up to isomorphism in $\QGr A_q$. When $q=0$, $A_q=k\<x,y\>/(xy)$ is a monomial algebra. However, unlike the case for degenerate Sklyanin algebras above, $A_0$ still has finite GK-dimension. Even though there are infinitely many point modules up to isomorphism in $\Gr A_0$, it will be shown there are only two up to isomorphism in $\QGr A_0$. 

Unlike the degenerate Sklyanin algebras, the structure of point modules for the degenerate quantum plane $A_0$ in some sense trivializes. The distinction seems to be that $A_0$ still has finite GK-dimension while the degenerate Sklyanin algebras have infinite GK-dimension. 

By \cite{HS0,HS1}, given any finitely presented monomial algebra (hereafter just monomial algebra), we can find a quiver $Q_A$ and a graded morphism $f:A \to kQ_A$ such that the functor $-\otimes_A kQ_A$ induces an equivalence of categories $\QGr A \equiv \QGr kQ_A$. Moreover, if $A$ has finite GK-dimension then so does $kQ_A$. 

If $M$ is a point module over $kQ_A$, it can be shown that $M$ is a point module when viewed as a module over $A$ via $f$. The only property to check is that $M$ is cyclic over $A$. If $A$ has finite GK-dimension, then it will be shown that every point module over $A$ comes from a point module over $kQ_A$ up to isomorphism in $\QGr A$. 

Since graded modules over path algebras can be viewed as graded representations of the quiver, it is useful to work with the algebras $kQ_A$ instead of $A$ itself. Hence, the rest of the paper will focus on path algebras of finite GK-dimension.  

If $kQ$ is a path algebra of finite GK-dimension, there is associated to each cyclic vertex $v$ of $Q$ a canonical point module which will be denoted by $\cO_v$. Here are the main results:

\begin{Thm}[See Theorem \ref{thm.simpleispoint}] \label{thm.main1Q}
Let $kQ$ be a path algebra of finite GK-dimension. Then every simple object of $\QGr kQ$ is isomorphic to a unique $\cO_v$ for some cyclic vertex $v$.
\end{Thm} 

Thus, the isomorphism classes of simple objects of $\QGr kQ$ are in bijection with the cyclic vertices of $Q$ of which there are only finitely many. In particular, the point modules $\cO_v$ give all the point modules up to isomorphism in $\QGr kQ$.

Outside of being in bijection with the simple objects of $\QGr kQ$, the cyclic vertices determine another structure which plays a big role in determining the structure of the category $\QGr kQ$.

\begin{Def}
Let $kQ$ be a path algebra of finite GK-dimension. The {\bf Ext-quiver} of $\QGr kQ$, denoted by $E_Q$, is defined by:
\begin{enumerate}
\item The vertices of $E_Q$ are the cyclic vertices of $Q$.
\item Let $v$ and $w$ be cyclic vertices and let $n$ and $m$ be the lengths of the simple cycles which contain $v$ and $w$ respectively. There is $1$ arrow $v\to w$ in $E_Q$ if there is a path in $Q$ from $v$ to $w$ whose length is a positive multiple of $nm$. 
\end{enumerate}
If $A$ is a monomial algebra of finite GK-dimension, define $E_A$ to be the quiver $E_{Q_A}$ where $Q_A$ is the Ufnarovskii graph of $A$.
\end{Def}

The next two theorems show the Ext-quiver is determined by $\QGr kQ$ and determines $\QGr kQ$ up to equivalence. 

\begin{Thm}[See Theorem \ref{thm.extofpoints}] \label{thm.main2Q}
Let $kQ$ be a path algebra of finite GK-dimension. Then given two cyclic vertices $v$ and $w$,
$$
\Ext^1_{\QGr kQ}(\cO_v,\cO_w)\neq 0
$$
if and only if there is an arrow from $v$ to $w$ in $E_Q$.
\end{Thm}

\begin{Thm} \label{thm.main3Q} 
Let $kQ$ and $kQ'$ be path algebras of finite GK-dimension. Then $\QGr kQ \equiv \QGr kQ'$ if and only if $E_Q=E_{Q'}$.
\end{Thm}

The proof of Theorem \ref{thm.main3Q} will be established in different sections. For the direction $\QGr kQ \equiv \QGr kQ' \Rightarrow E_Q=E_{Q'}$, see the remarks directly after Theorem \ref{thm.extofpoints}. For the direction $E_Q=E_{Q'} \Rightarrow \QGr kQ \equiv \QGr kQ'$, see the remarks after Theorem \ref{thm.k0qgrkQ}.

Let $(\cP,\preceq)$ be a finite poset. We can consider $\cP$ as a quiver, which will also be denoted $\cP$, in the following way: the vertices of $\cP$ are just the elements while there is an arrow $x\to y$ if and only if $x\prec y$. Such quivers will be called {\bf poset quivers}. It will be shown that for any quiver $Q$ for which $kQ$ has finite GK-dimension, the Ext-quiver $E_Q$ is a poset quiver. That is to say, we can view (the vertices of) $E_Q$ as a poset by defining $v\prec w$, for cyclic vertices $v$ and $w$, if there is an arrow $v\to w$ in $E_Q$. See section \ref{sec.extposet} for details.

Let $\cP$ be a finite poset. Define $\Gamma(\cP)$ to be the quiver whose vertices are the elements of $\cP$ and where there is an arrow $x\to y$ if and only if $x\preceq y$. $\Gamma(\cP)$ is nothing more than the quiver $\cP$ except now there is a loop placed at each vertex. The path algebra $k\Gamma(\cP)$ has finite GK-dimension and $E_{\Gamma(\cP)}=\cP$. Hence, we get the following interesting corollary to Theorem \ref{thm.main3Q}:

\begin{Cor}
If $A$ is either a monomial algebra or a path algebra of finite GK-dimension, then there is a finite poset $\cP$ and an equivalence
$$
\QGr A \equiv \QGr k\Gamma(\cP).
$$
\end{Cor}
\begin{proof}
Just take $\cP$ to be the poset $E_A$ and apply Theorem \ref{thm.main3Q}.
\end{proof}

This corollary says that the path algebras $k\Gamma(\cP)$ for $\cP$ a finite poset form a class of canonical noncommutative homogeneous coordinate rings for the noncommutative projective schemes $\Proj_{nc}A$ where $A$ is a monomial algebra or path algebra of finite GK-dimension.

\subsection{Acknowledgements} The author would like to thank Chris McMurdie, Gautam Sisodia and S. Paul Smith for useful conversations as well as S. Paul Smith for reading an early version of this paper and providing comments.

\section{Notation and Conventions.}

Throughout, $k$ is a fixed field. 

\subsection{Quivers.}
$Q$ will always denote a quiver (i.e., directed graph) with a finite number of vertices and arrows. The set of vertices and arrows will be denoted $Q_0$ and $Q_1$ respectively while $s,t:Q_1 \to Q_0$ will be the source and target maps.

A path in $Q$ is an ordered tuple of vertices and arrows
$$
p=(v_0,a_1,v_1,\ldots,v_{m-1},a_m,v_m)
$$
where $v_{i-1}=s(a_i)$ and $v_i=t(a_i)$. It is more common to use the shorthand notation $p=a_1a_2\cdots a_m$. If $p$ is a path then the source of $p$ is the source of $a_1$ and the target of $p$ is the target of $a_m$. Below is some common terminology used in later sections.

\begin{Def} \label{def.pathterms}
Let $p=(v_0,a_1,v_1,\cdots,a_m,v_m)$ be a path in $Q$.
\begin{enumerate}
\item $p$ is a {\bf chain} if $a_i\neq a_j$ for $i\neq j$.
\item $p$ is a {\bf simple chain} if $v_i\neq v_j$ for $i\neq j$. 
\item $p$ is {\bf closed}, or is a closed path, if $v_0=v_m$.
\item $p$ is a {\bf cycle} if it is a closed chain.
\item $p$ is a {\bf simple cycle} if $v_0=v_m$ but $v_i\neq v_j$ otherwise.
\end{enumerate}
A vertex is called {\bf cyclic} if it is part of some cycle, otherwise it is {\bf acyclic}. A quiver is called {\bf acyclic} if it contains no cycles. An arrow with the same source and target is called a {\bf loop}. 
\end{Def}

\begin{Rmk}
A simple cycle is necessarily a chain. A cycle is simple if and only if no two distinct arrows in the cycle have the same source if and only if no two distinct arrows in the cycle have the same target. Pictorially, a simple cycle is anything of the form
$$
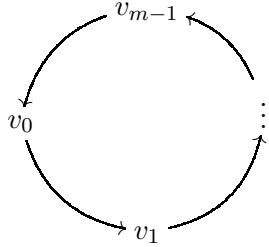
\UseComputerModernTips
\xymatrix{ {} & {v_{m-1}}\ar@/_1pc/[ld] & {} & {} \\
           {v_0}\ar@/_1pc/[rd] & {} & {\vdots}\ar@/_1pc/[ul] & {} \\
           {} & {v_1}\ar@/_1pc/[ur] & {} & {} }
$$
\end{Rmk}

\subsection{Path algebras.}
$kQ$ will denote the path algebra of $Q$. For each vertex $v$, let $e_v=(v)$ denote the trivial path at that vertex. A basis for $kQ$ is given by all paths in $Q$, including trivial paths, and multiplication of two paths is given by concatenation. 

More explicitly, if $p=(v_0,a_1,\ldots,a_m,v_m)$ and $q=(u_0,b_1,\ldots,b_n,u_n)$, then $pq=(v_0,a_1,\ldots,a_m,v_m=u_0,b_1,\ldots,b_n,u_n)$ if $v_m=u_0$ while $pq=0$ if $v_m\neq u_0$.  

The algebra $kQ$ is given the natural grading where each trivial path has degree $0$ and each arrow has degree $1$.

\subsection{The quotient category.}
If $A$ is an $\NN$-graded $k$-algebra, we write $\Gr A$ for the category of $\ZZ$-graded right $A$ modules. Given a graded module $M$, an element $m\in M$ is called {\bf torsion} if $mA_{\geq n}=0$ for some $n$. A module $M$ is called {\bf torsion} if every element of $M$ is torsion and is {\bf torsion free} if no non-zero element is torsion.

The full subcategory of all torsion modules is denoted $\Tors A$. The quotient of $\Gr A$ by $\Tors A$ is labeled $\QGr A$ and we let
$$
\pi^*:\Gr A \to \QGr A
$$ 
denote the canonical quotient functor. $\Tors A$ is a localizing subcategory, that is, $\pi^*$ has a right adjoint which will be denoted by $\pi_*$.

Every graded module $M \in \Gr A$ has a largest submodule contained in $\Tors A$ which is denoted $\tau M$. Moreover, $M/\tau M$ is torsion free and $\pi^*M\cong \pi^* (M/\tau M)$.

Two graded modules $M$ and $N$ are called {\bf tails equivalent} if $M_{\geq n}\cong N_{\geq n}$ for some $n\in \ZZ$. If $M$ and $N$ are tails equivalent, then $\pi^*M\cong \pi^*N$ in $\QGr A$.

In the case where $A$ is right graded coherent, the category $\gr A$ of all finitely presented $\ZZ$-graded right $A$-modules is an abelian subcategory of $\Gr A$. The subcategory $\tors A:=\gr A \cap \Tors A$ is a Serre subcategory of $\gr A$. The inclusion functor $\gr A \to \Gr A$ induces a fully faithful functor
$$
\qgr A:=\gr A/\tors A \to \QGr A.
$$
which is an equivalence between $\qgr A$ and the finitely-presented objects of $\QGr A$.  

\subsection{Point modules.}
\begin{Def}
Let $A$ be an $\NN$-graded $k$-algebra generated by $A_1$ over $A_0$. A graded right module $M=\oplus M_i$ is a {\bf point module} if
\begin{itemize}
\item $M=M_0A$, 
\item $\dim_k M_i=1$ for all $i\geq 0$.
\end{itemize}
\end{Def}

Point modules determine simple objects in $\QGr A$ although not every simple object comes from a point module in general. A countereample is the free algebra on two generators, or more generally, a path algebra of infinite GK-dimension.

\section{Graded modules and graded representations.}
\subsection{}
Given a quiver $Q$, the category $\Gr kQ$ is equivalent to the category $\GrRep kQ$ of graded representations of the quiver $Q$. A graded representation of $Q$, denoted $(M_v,M_a)$, is the assignment of a graded vector space $M_v$ ($k$ is in degree $0$) to each vertex $v$ and for each arrow $a$ a linear map $M_a : M_{s(a)}\to M_{t(a)}$ of degree one. A morphism $\phi:(M_v,M_a) \to (N_v,N_a)$ of graded representations is a collection of graded vector space maps $\phi_v:M_v\to N_v$ such that for each arrow $a$, the diagram
$$
\UseComputerModernTips
\xymatrix{ {M_{s(a)}}\ar[r]^{M_a} \ar[d]_{\phi_{s(a)}} & {M_{t(a)}} \ar[d]^{\phi_{t(a)}} \\
           {N_{s(a)}}\ar[r]_{N_a} & {N_{t(a)}} }
$$
commutes.

The equivalence is determined by sending a graded module $M$ to the data $(Me_v,M_a)$ where $M_a$ is the degree $1$ linear map determined by right multiplication by $a$.

Let $V = \{v_1 \ldots v_n\}\subset Q_0$. If $M$ is a module over $kQ$ and $m\in M$ with $m =\sum_{v_i\in V}me_{v_i}$, then we say $m$ is {\bf supported} on the set $V$ . In particular, if $m = me_v$ for some vertex $v$ then $m$ is supported on $v$. If every element of $M$ is supported on $V$ then we say $M$ is supported on $V$.

\section{Path algebras of finite GK-dimension.}
\subsection{}
In \cite{U1}, V. Ufnarovskii gives a criterion which allows one to determine the growth of a quiver, which is the same as the growth of the path algebra, based on a simple property of the quiver. 

Let $p=(v_0,a_1,\cdots,a_m,v_m)$ be a path in $Q$. Define $Q(p)=(Q(p)_0,Q(p)_1)$ to be the subquiver of $Q$ consisting of all the vertices and arrows that make up $p$, i.e, $Q(p)_0=\{v_0,\ldots,v_m\}$ and $Q(p)_1=\{a_1,\ldots,a_m\}$. Let $C$ be a subquiver of $Q$, call $C$ a {\bf chain}({\bf simple chain, closed, cycle, simple cycle}) if $C=Q(p)$ where $p$ is a chain (simple chain, closed, cycle, simple cycle).

Two cycles overlap if they have a vertex in common. Let $v\in Q_0$ be a vertex. If there are simple cycles $p_1$ and $p_2$ such that $Q(p_1)\neq Q(p_2)$ but $v\in Q(p_1)_0\cap Q(p_2)_0$, then $v$ is called {\bf doubly cyclic}.

Let $Q$ be a quiver with subquivers $C_1$ and $C_2$ which are simple cycles. Define $C_1\preceq C_2$ if there is a simple chain from a vertex of $C_1$ to a vertex of $C_2$ or if $C_1$ and $C_2$ share a common vertex. This makes the set of simple cycles in $Q$, which is denoted by $\mC(Q)$, a finite preorder.  

\begin{lemma}
Let $Q$ be a quiver. The preorder $\mC(Q)$ is a poset if and only if $Q$ has no doubly cyclic vertices.
\end{lemma}
\begin{proof}
If $Q$ has a doubly cyclic vertex then there are distinct simple cycles $C_1$ and $C_2$ having a common vertex. Hence,  $C_1\preceq C_2$ and $C_2\preceq C_1$ showing $\mC(Q)$ is not a poset.

Suppose $\mC(Q)$ is not a poset. Then there are two distinct simple cycles $C_1$ and $C_2$ such that either $C_1$ and $C_2$ share a vertex or there is a simple chain from a vertex of $C_1$ to a vertex of $C_2$ and a simple chain from a vertex of $C_2$ to a vertex of $C_1$. In both cases, $Q$ has a doubly cyclic vertex. 
\end{proof}

The following Theorem is in \cite{U1}, though not stated in the following manner.

\begin{Thm} \label{thm.ufresult}
Let $Q$ be a quiver and $\mC(Q)$ the associated preorder. If $Q$ has a doubly cyclic vertex then $kQ$ has exponential growth. Otherwise, $kQ$ has polynomial growth of degree $d$ where $d$ is the cardinality of a largest totally ordered subset of $\mC(Q)$.
\end{Thm}

Hence, the path algebra of a quiver $Q$ has finite GK-dimension if and only if there are no doubly cyclic vertices. In this case, the GK-dimension is $d$ where $d$ is the size of a largest totally ordered subset of $\mC(Q)$.

\begin{example}
The quivers
$$
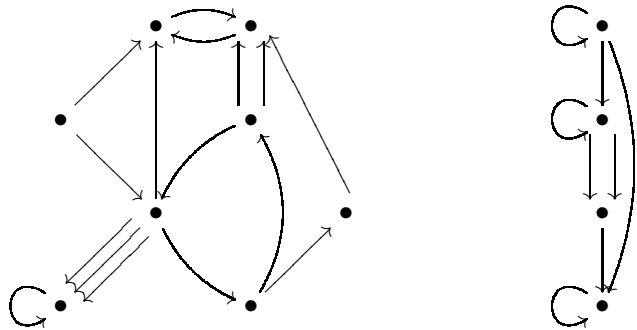
\UseComputerModernTips
\xymatrix{ {} & {\bullet}\ar@/^/[r] & {\bullet}\ar@/^/[l] & {} & {} & {} & {\bullet}\ar[d] \ar@(ul,dl)[] \ar@/^1pc/[ddd] & {} \\
           {\bullet}\ar[ur] \ar[dr] & {} & {\bullet}\ar@/_/[dl] \ar@<1ex>[u] \ar@<-1ex>[u] & {} & {} & {} & {\bullet}\ar@<1ex>[d] \ar@<-1ex>[d] \ar@(ul,dl)[] & {}  \\
           {} & {\bullet}\ar@/_/[dr] \ar[dl] \ar@<1ex>[dl] \ar@<-1ex>[dl] \ar[uu] & {} & {\bullet} \ar@<-1ex>[uul] & {} & {} & {\bullet}\ar[d] & {} \\
           {\bullet}\ar@(ul,dl)[] & {} & {\bullet}\ar@/_1pc/[uu] \ar[ur] & {} & {} & {} & {\bullet}\ar@(ul,dl)[] & {}}
$$
have finite GK dimension. The first has dimension 2 while the second has dimension 3. The following quivers have infinite GK-dimension.
$$
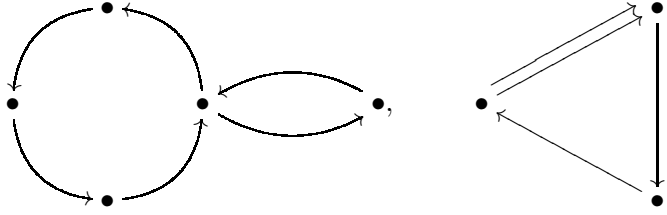
\UseComputerModernTips
\xymatrix{ {} & {\bullet}\ar@/_1pc/[dl] & {} & {} & {} & {} & {} & {\bullet}\ar[dd] \\
           {\bullet}\ar@/_1pc/[dr] & {} & {\bullet}\ar@/_1pc/[rr] \ar@/_1pc/[ul] & {} & {\bullet,}\ar@/_1pc/[ll] & {\bullet}\ar@<1ex>[urr] \ar[urr] & {} & {} \\
           {} & {\bullet}\ar@/_1pc/[ur] & {} & {} & {} & {} & {} & {\bullet}\ar[ull] }
$$
\end{example}

\begin{example}
As the path algebras of the first two quivers in the previous example have finite GK-dimension, their Ext-quivers are defined. The Ext-quivers are given below in the same ordering as the quivers above:
$$
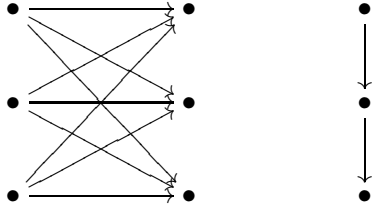
\UseComputerModernTips
\xymatrix{ {\bullet} \ar[rr] \ar[drr] \ar[ddrr] & {} & {\bullet} & {} & {\bullet} \ar[d] \\
           {\bullet} \ar[urr] \ar[rr] \ar[drr] & {} & {\bullet} & {} & {\bullet} \ar[d] \\
					 {\bullet} \ar[uurr] \ar[urr] \ar[rr] & {} & {\bullet} & {} & {\bullet} }
$$
\end{example}
Below are a few lemmas about quivers of finite growth that will be used implicitly throughout.

\begin{lemma}
Suppose $p$ is a closed path such that
\begin{enumerate}
\item $p$ is not a simple cycle,
\item $p\neq q^n$ for any closed path $q$ and $n\geq 2$.
\end{enumerate}
Then there are two distinct  arrows in $p$  with the same source.
\end{lemma}
\begin{proof}
Write $p=(v_0,a_1,\ldots,a_m,v_m)$. Suppose distinct arrows in $p$ have distinct sources but $p$ is not a simple cycle. As $p$ is not simple there are natural numbers $i,j$ with $0\leq i<j\leq m$ and $(i,j)\neq (0,m)$ such that $v_i=v_j$. Choose such a pair $(i,j)$ such that $j-i$ is minimal. As $p$ is closed we may assume $i=0$. We can write $p=qr$ where $q=(v_0,a_1,\ldots,a_j,v_j)$ and $r=(v_j,a_{j+1},\ldots,a_m,v_m)$. By the choices made, $q$ is a simple cycle. As $p$ is not simple we know $r$ is not a trivial path. 

Since $s(a_{j+1})=v_j=v_0=s(a_1)$ and distinct arrows have distinct sources it follows that $a_{j+1}=a_1$. Therefore, $v_1=t(a_1)=t(a_{j+1})=v_{j+1}$. By similar reasoning as for $a_1$ and $a_{j+1}$ we deduce $a_2=a_{j+2}$ and by induction $a_{j+l}=a_l$ for all $1\leq l\leq j$. Therefore, $p=q^2r'$ where $r'=(v_{2j},a_{2j+1},\ldots,a_m,v_m)$. By induction we can continue this to write $p=q^n$ where $n=m/j>1$.
\end{proof}

\begin{lemma} \label{lem.simpleclosedsubs}
A quiver $Q$ has no doubly cyclic vertices if and only if every closed subquiver of $Q$ is a simple cycle.
\end{lemma}
\begin{proof}
$(\Leftarrow)$ Suppose $v\in Q_0$ is a doubly cyclic vertex. Then there are distinct simple cycles $p_1=(v_0,a_1,\ldots,a_m,v_m)$ and $p_2=(u_0,b_1,\ldots,b_n,u_n)$ with $Q(p_1)\neq Q(p_2)$ but for which $v_i=u_j$. As $p_1$ and $p_2$ are cycles, we may assume $i=j=0$. Consider the closed path 
$$
p_1p_2=(v_0,a_1,\ldots,a_m,v_m=v_0=u_0,b_1,\ldots,b_n,u_n)
$$
and let $C=Q(p_1p_2)$. Then $C$ is a closed subquiver which is not a simple cycle.

$(\Rightarrow)$ Suppose $C$ is a closed subquiver which is not a simple cycle. Write $C=Q(p)$ where $p=(v_0,a_1,\ldots,a_m,v_m)$ is a closed path. Since $Q(p^n)=Q(p)$ we can assume $p\neq q^n$ where $q$ is closed and $n\geq 2$. Since $p$ is a closed path which is not simple and is not a power, the previous lemma says there are two distinct arrows in $p$ with the same source $u$. We may assume $u=v_0$ and write $p=qr$ where $q=(v_0,a_1,\ldots,a_l,v_l=v_0)$ and $r=(v_l,a_{l+1},\ldots,a_m,v_m)$ with $a_{l+1}\neq a_1$.

If $q$ is not a simple cycle, then we can find vertices $v_i$ and $v_j$ in $q$ such that $0 \leq i < j \leq l$ and $(i,j) \neq (0,l)$ with $v_i = v_j$. If $i=0$, then instead of looking at the pair $(0,j)$ look at the pair $(j,l)$. Consider the path $q_1=(v_0,a_1,\ldots,a_i,v_i,a_{j+1},\ldots,v_l)$ which is obtained by removing the subpath from $v_i$ to $v_j$ in $q$. Notice $q_1$ is a closed path of strictly smaller length than $q$ and still contains the arrow $a_1$. As the length of the path decreases, we can continue this process only finitely many times.

The only way the process can stop is if we eventually obtain a closed path $q'=q_n$ which is a simple cycle. This simple cycle starts at the vertex $v_0$ and contains the arrow $a_1$. Similarly, we can do the same process to $r$ to obtain a simple cycle $r'$ which starts at the vertex $v_l=v_0$ and contains the arrow $a_{l+1}\neq a_1$. Hence, $Q(q')$ and $Q(r')$ are distinct simple cycles which contain $v_0$ showing $v_0$ is a doubly cyclic vertex. 
\end{proof}

\section{Cyclic point modules and the simple objects of $\QGr kQ$.} \label{sec.cyclicpoints}

Suppose $Q$ has no doubly cyclic vertices and let $v$ be a cyclic vertex. There is a special point module associated to the cyclic vertex $v$, which is a quotient of the module $e_vkQ$, and will be denoted by $\cO_v$. These special point modules will be called {\it cyclic point modules}.

In words, $\cO_v$ is the quotient of $e_vkQ$ by the right sub-module spanned by all paths beginning at $v$ but which end at a vertex $u$ not in the cycle $p$.

Here is a precise description of $\cO_v$. 

Let $p=(v=v_0,a_1,\ldots,a_n,v_n=v)$ be the simple cycle which starts at $v$. As $kQ$ has finite GK-dimension, $a_{i+1}$ is the only arrow from $v_i \to v_{i+1}$. Only considering arrows which start at a vertex in $p$, the quiver locally looks like 
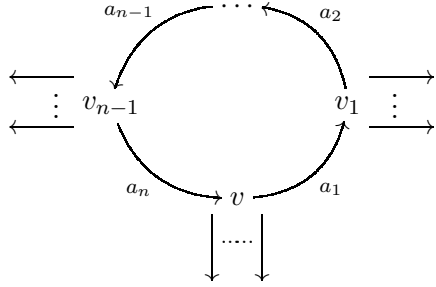
\begin{equation} \label{pict.simplecycle}
\UseComputerModernTips
\xymatrix{ {} & {} & {\cdots}\ar@/_1pc/[dl]_{a_{n-1}} & {} & {} \\
           {} & {v_{n-1}} \ar@/_1pc/[dr]_{a_n} \ar@<2ex>_{\vdots}[l] \ar@<-2ex>[l] & {} & {v_1} \ar@<2ex>[r] \ar@<-2ex>[r]^{\vdots} \ar@/_1pc/[ul]_{a_2} & {} & {} \\
           {} & {} & {v}\ar@/_1pc/[ur]_{a_1} \ar@<2ex>[d]_{\hdots} \ar@<-2ex>[d]^{\hdots} & {} & {} & {} \\
           {} & {} & {} & {} & {} & {} }
\end{equation} 

Every path in $Q$ which has source $v$ has one of the forms:
\begin{enumerate}
\item[$\bullet$] $p^ma_1\cdots a_i$ such that $m\in \NN$ and $0\leq i<n$ ($i=0$ means just $p^m$),
\item[$\bullet$] $p^ma_1\cdots a_ibq$ where $m\in \NN$, $0\leq i<n$, $b\neq a_{i+1}$ is an arrow starting at $v_i$ (if any) and $q$ is any path which begins at $t(b)$.  
\end{enumerate}

Consider all the submodules $p^ma_1\cdots a_ibkQ$ of $e_v kQ$ where $b\neq a_{i+1}$ is an arrow which starts at $v_i$. Since the paths in $Q$ form a basis for $kQ$ we get
$$
\sum{p^ma_1\cdots a_ibkQ}=\bigoplus p^ma_1\cdots a_ibkQ.
$$
where the sums are over all $(m,i,b)$ with $m\geq 0$, $0\leq i<n$, and $b$ an arrow not in $p$ starting at $v_i$.

$\cO_v$ is the quotient of $e_vkQ$ by the submodule $\bigoplus p^ma_1\cdots a_ibkQ$. By definition of $\cO_v$, the following sequence is exact:
\begin{equation} \label{eq.cyclicpoint}
\UseComputerModernTips
\xymatrix{ {0}\ar[r] & {\bigoplus p^ma_1\cdots a_ibkQ}\ar[r]^(0.7){\i} & {e_vkQ}\ar[r] & {\cO_v}\ar[r] & {0} }
\end{equation}
The map $\i$ is simply the inclusion map. 

Since the modules $e_vkQ$ and $p^ma_1\cdots a_ibkQ$ are projective in $\Gr kQ$, the exact sequence \ref{eq.cyclicpoint} is a projective resolution of $\cO_v$ in $\Gr kQ$. 

In \cite{Sm1}, Smith proves $\pi^*kQ$ is a projective object in $\QGr kQ$. Hence, as $\pi^*$ preserves coproducts, $\pi^*e_vkQ$ is a projective object for any vertix $v$ in $Q$. Since the quotient functor $\pi^*:\Gr kQ \to \QGr kQ$ is exact, the sequence \ref{eq.cyclicpoint} determines a projective resolution of $\cO_v$ in $\QGr kQ$. 

As a vector space, $\cO_v$ has a basis consisting of all paths of the form $p^ma_1\cdots a_i$. If $b$ is any arrow in $Q$, then in $\cO_v$, 
$$
p^ma_1\cdots a_i.b=\left\{
\begin{array}{rl}
p^ma_1\cdots a_ia_{i+1} & \text{ if }b=a_{i+1} \\
0 & \text{ otherwise.}
\end{array}
\right.
$$
If $j\in \NN$, then we can uniquely write $j=mn+i$ for some $i<n$ and we get $(\cO_v)_j=kp^ma_1\cdots a_i$. In particular, $(\cO_v)_0=ke_v$. 

\begin{proposition} 
Let $kQ$ be a path algebra of finite GK-dimension. For each cyclic vertex $v$, $\cO_v$ is a point module.
\end{proposition}
\begin{proof}
It was already noted that $\dim (\cO_v)_j=1$ for all $j$. Also, since $(\cO_v)_0=ke_v$ and $p^ma_1\cdots a_i=e_vp^ma_1\cdots a_i$, we get $\cO_v=(\cO_v)_0kQ$. 
\end{proof}

\begin{lemma} \label{lem.cyclicpointsarediff} 
Let $v$ and $w$ be cyclic vertices. Then $\Hom_{\QGr kQ}(\cO_v,\cO_{w})\neq 0$ if and only if $v=w$. In particular, $\pi^*\cO_v \not\cong \pi^*\cO_w$ if $v\neq w$.
\end{lemma}
\begin{proof}
If $v$ and $w$ are distinct cyclic vertices, then $(\cO_v)_n$ and $(\cO_w)_n$ are supported at different vertices for all $n$. Hence, the only graded morphism from $(\cO_v)_{\geq n}\to \cO_{w}$ for all $n$ is zero. Hence,
$$
\Hom_{\QGr kQ}(\cO_v,\cO_{w})=\varinjlim_{n}\Hom_{\Gr kQ}((\cO_v)_{\geq n},\cO_{w})=0.
$$  
\end{proof}

\begin{proposition} \label{prop.shiftofpoint}
Let $v$ be a cyclic vertex and write
$$
p=(v_0=v,a_1,v_1,\ldots,a_n,v_n=v)
$$ 
for the simple cycle which contains $v$. Then in $\Gr kQ$, $\cO_v(1)_{\geq 0}\cong \cO_{v_1}$. 
\end{proposition}
\begin{proof}
The simple cycle which contains $v_1$ is $p'=(v_1,a_2,\ldots,a_n,v_0,a_1,v_1)$. Notice $a_1(p')^n=p^na_1$ for all $n$. Define a linear map $\phi:\cO_{v_1} \to \cO_v$ by sending $(p')^na_2\cdots a_i$ to $a_1(p')^na_2\cdots a_i=p^na_1\cdots a_i$. It is easy to check this defines an injective graded module map $\phi:\cO_{v_1} \to \cO_v(1)$. Since $e_{v_1}\mapsto a_1$, the image of $\phi$ is $\cO_v(1)_{\geq 0}$. Hence, $\cO_{v_1} \cong \cO_{v}(1)_{\geq 0}$. 
\end{proof}

\begin{Cor} \label{cor.shiftofpoint}
Let $v$ be a cyclic vertex. If $z\in \ZZ$, then in $\QGr kQ$, $\pi^*\cO_v(z)\cong \pi^*\cO_{v'}$ for some cyclic vertex $v'$ in $p$.
\end{Cor}
\begin{proof}
Write $p=(v_0=v,a_1,v_1,\ldots,v_n=v)$ for the simple cycle which begins at $v$. By Proposition \ref{prop.shiftofpoint} $\cO_v(1)_{\geq 0}\cong \cO_{v_1}$ which implies
$$
\pi^*\cO_v(1)\cong \pi^*\cO_v(1)_{\geq 0} \cong \pi^*\cO_{v_1}.
$$

On the other hand, $\cO_{v_{n-1}}(1)_{\geq 0} \cong \cO_v$ so we get
$$
\pi^*\cO_{v_{n-1}}(1)\cong \pi^*\cO_{v_{n-1}}(1)_{\geq 0} \cong \pi^*\cO_v
$$
which shows 
$$
\pi^*\cO_v(-1)\cong \pi^*\cO_{v_{n-1}}.
$$
Hence, the corollary is finished by induction.
\end{proof}

\begin{proposition} \label{prop.everypointiscyclic}
Let $P$ be a point module over a path algebra of finite GK-dimension. There is a cyclic vertex $v$ such that $\pi^*P \cong \pi^*\cO_v$.
\end{proposition}
\begin{proof}
Let $P$ be any point module. For each $i$, there is a unique vertex $v_i$ such that $P_ie_{v_i}=P_i$ and $P_ie_u=0$ for all $u\neq v_i$. Hence, we have an infinite sequence of vertices $\bv=(v_0,v_1,\ldots)$. Let $v_i$ and $v_{i+1}$ be two vertices in $\bv$. Since $P_{i+1}=P_ikQ_1$, there must be an arrow from $v_i$ to $v_{i+1}$. Hence, the sequence $\bv$ is the sequence of vertices for some infinite path $(v_0,a_1,v_1,\ldots)$ in $Q$. As $kQ$ has finite GK-dimension, every infinite path in $Q$ must be of the form $qp^{\infty}$ where $q$ is a finite path, $p=(u_0,b_1,\ldots,b_m,u_m)$ is a simple cycle and $p^{\infty}$ is the infinite path which just continually loops around $p$.

Let $n$ be the length of the path $q$ and consider the point module $P':=P_{\geq n}(n)$. The infinite sequence of vertices associated to $P'$ is just 
$$
(u_0,u_1,\ldots,u_m=u_0,u_1,\ldots).
$$ 
Hence, the only arrow which does not annihilate $P'_{lm+i}$ is the arrow $a_{i+1}:u_{i}\to u_{i+1}$ since $a_{i+1}$ is the only arrow from $u_{i}$ to $u_{i+1}$. Hence, it can be seen that $P'$ is isomorphic to the cyclic point module $\cO_{u_0}$.   

As $P_{\geq n}(n) \cong \cO_{u_0}$, we get 
$$
\pi^*P(n) \cong \pi^*P_{\geq n}(n) \cong \pi^*\cO_{u_0} \Rightarrow \pi^*P \cong \pi^*\cO_{u_0}(-n)
$$
By Corollary \ref{cor.shiftofpoint}, $\pi^*\cO_{u_0}(-n)\cong \pi^*\cO_{v}$ for some cyclic vertex $v$ which shows $\pi^*P \cong \pi^*\cO_v$ for a cyclic vertex $v$.
\end{proof}

\begin{proposition} \label{prop.high1sispoint}
Suppose $kQ$ has finite GK-dimension. Let $M$ be a graded right $kQ$-module such that $\dim M_j=1$ for $j\gg 0$. Then $\pi^*(M)\neq 0$ if and only if there is a cyclic vertex $w$ such that $\pi^*(M) \cong \pi^*\cO_w$.
\end{proposition}
\begin{proof}
Suppose $\pi^*(M)$ is not zero. Find $n'\in \NN$ such that $\dim(M_i)=1$ for $i\geq n'$. The object $\pi^*(M_{\geq n'})$ is also nonzero which implies there is a homogeneous element $m\in M_{n}$, with $n\geq n'$, that is not torsion. Hence, for every positive integer $i$, there is a path $p$ of length $i$ that does not kill $m$. Therefore, $mkQ_i$ is a nonzero subspace of $M_{n+i}$ which implies $mkQ_{i}=M_{n+i}$ as $\dim M_{n+i}=1$. Hence, $M_{\geq n}=mkQ$. 

Since the module $M_{\geq n}(n)$ is generated in degree zero and satisfies 
$$
\dim (M_{\geq n}(n))_i=1
$$ 
for all $i>0$, $M_{\geq n}(n)$ is a point module. Hence, by Proposition \ref{prop.everypointiscyclic}, there is a cyclic vertex $v$ such that $\pi^*M_{\geq n}(n) \cong \pi^*\cO_v$. Thus, $\pi^*M_{\geq n} \cong \pi^*\cO_v(-n) \cong \pi^*\cO_w$ for some cyclic vertex $w$ from which it follows that 
$$
\pi^*M \cong \pi^*M_{\geq n} \cong \pi^*\cO_w.
$$ 
\end{proof}

The following theorem shows that every non-zero object in $\QGr kQ$ contains a point module.  

\begin{Prop} \label{prop.hasapoint}
Let $kQ$ be a path algebra of finite GK-dimension. For any $M\in \Gr kQ$, $\pi^*(M)=0$ if and only if 
$$
\Hom_{\QGr kQ}(\pi^*\cO_v,\pi^*(M))=0
$$
for every cyclic vertex $v$.
\end{Prop}
\begin{proof}
Clearly $\pi^*(M)=0$ implies all the $\Hom$ spaces are zero. Suppose $\pi^*(M)$ is not a zero object. By replacing $M$ with $M/\tau M$ we may assume $M$ is torsion free.

Let $C$ be a simple cycle in $\mC(Q)$ maximal with respect to containing a vertex in the support of $M$ and denote such a vertex by $v_0$. Pick a nonzero homogeneous $m\in M$ such that $m=me_{v_0}$. Write $C=Q(p)$ where $p=(v=v_0,a_1\cdots a_n,v_n=v)$. 

Suppose there is a path of the form $p^la_1\cdots a_ib$, with $b\neq a_{i+1}$, for which 
$$
m':=mp^la_1\cdots a_ib\neq 0.
$$ 
Since $m'$ is not torsion, there are paths of arbitrarily high degree which do not annihilate it. Hence, we can find a path $q'$ from $t(b)$ to a cyclic vertex $v'$ in a simple cycle $C'\neq C$ such that $m'q'=mqp^la_1\cdots a_ibq'\neq 0$. This is a contradiction however as $C\prec C'$ and $C'$ containing a vertex in the support of $M$ implies $C$ is not maximal with this property. Therefore, $mp^la_1\cdots a_ib=0$ for all such paths. From this we find the only elements in $kQ$ which do not necessarily annihilate $m$ are those of the form $p^la_1\cdots a_i$. However, as $m$ is not torsion, none of the elements $p^la_1\cdots a_i$ annihilate $m$. Hence, $mkQ$ is a nonzero submodule which has dimension one in high degree. 

Thus, we have a submodule $mkQ$ of $M$ which has dimension $1$ in high degree and $\pi^*(mkQ)\neq 0$. By Proposition \ref{prop.high1sispoint}, there is a point module $\cO_{w}$ such that $\pi^*(mkQ)\cong \pi^*\cO_w$. Thus,
$$
\Hom_{\QGr kQ}(\pi^*\cO_w,\pi^*M)\neq 0.
$$   
\end{proof}

We can now prove Theorem \ref{thm.main1Q}.

\begin{Thm}[Theorem \ref{thm.main1Q}.] \label{thm.simpleispoint}
Let $kQ$ be a path algebra of finite GK-dimension. The objects 
$$
\{\cO_v\;|\;v \text{ is a cyclic vertex}\}
$$
form a complete set of representatives of the isomorphism classes of simple objects in $\QGr kQ$.
\end{Thm}
\begin{proof}
Since $\cO_v$ is a point module, $\cO_v$ is a simple object in $\QGr kQ$ for each cylic vertex $v$. By Lemma \ref{lem.cyclicpointsarediff}, $\cO_v$ and $\cO_w$ are not isomorphic unless $v=w$. If $S$ is any module for which $\pi^*S$ is a simple object, then by Proposition \ref{prop.hasapoint}, 
$$
\Hom_{\QGr kQ}(\pi^*\cO_v,\pi^*S)\neq 0
$$ 
for some cyclic vertex $v$ which implies $\pi^*\cO_v \cong \pi^*S$.
\end{proof}

\begin{corollary} \label{cor.cyclicvertispreserved}
Let $kQ$ and $kQ'$ be path algebras of finite GK-dimension. If $\QGr kQ\equiv \QGr kQ'$, then $Q$ and $Q'$ have the same number of cyclic vertices. 
\end{corollary}
\begin{proof}
This follows from the fact that the number of cyclic vertices is precisely the number of simple objects in $\QGr kQ$.
\end{proof}

\section{Extensions between point modules over path algebras of finite GK-dimension.}

\subsection{} Let $kQ$ be a path algebra of finite GK-dimension. In this section a condition is given which determines when $\Ext^1_{\QGr kQ}(\cO_v,\cO_w)\neq 0$ for cyclic vertices $v$ and $w$. 

Let $v$ and $w$ be cyclic vertices. To compute $\Ext^1_{\QGr kQ}(\cO_v,\cO_w)$ we can use the projective resolution used to define $\cO_v$:
$$
0 \to \bigoplus p^ma_1\cdots a_ibkQ  \to e_vkQ \to \cO_v \to 0.
$$
The notation follows that developed in Section \ref{sec.cyclicpoints}.

Using the long exact sequence associated to $\Ext$ and the fact that $\pi^*e_vkQ$ is projective we get an exact sequence
\begin{eqnarray} \label{seq.extcomp}
0\to \Hom_{\QGr kQ}(\cO_v,\cO_w) \to \Hom_{\QGr kQ}(e_vkQ,\cO_w)\to  \\
\to \Hom_{\QGr kQ}\left(\bigoplus (p^ma_1\cdots a_ibkQ),\cO_w\right)\to \Ext^1_{\QGr kQ}(\cO_v,\cO_w)\to 0. \nonumber
\end{eqnarray}
This exact sequence will be useful to determine when $\Ext^1(\cO_v,\cO_w)=0$. However, in the case where there are non-trivial extensions, we will see how to construct a large family of them explicitly.  

It will be useful to work over Veronese subalgebras when computing $\Ext^1(\cO_v,\cO_w)$. Hence, the next subsection recalls the relationship between an algebra and its Veronese subalgebras.

\subsection{Veronese subalgebras of path algebras.}
This subsection recalls a theorem of A.B. Verevkin which will be helpful in computing $\Ext^1(\cO_v,\cO_w)$.

Let $A$ be a locally finite graded $k$-algebra generated by $A_1$ over $A_0$ and for $d\in \NN\-\{0\}$, let $A^{(d)}=\oplus_{i\geq 0} A_{id}$ be the $d$-th Veronese subalgebra. Define $V:\Gr A\to \Gr A^{(d)}$ to be the functor which takes $M\in \Gr A$ to $V(M)\in \Gr A^{(d)}$ where $V(M)_i=M_{i\cdot d}$. Conversely, define $T:\Gr A^{(d)}\to \Gr A$ by $T(N)=N\otimes_{A^{(d)}}A$ with grading
$$
(N\otimes_{A^{(d)}}A)_j=\sum_{n\cdot d+l=j}N_n\otimes_{A^{(d)}}A_l.
$$

\begin{Thm}[\cite{ABV} Section 4.] \label{thm.ABV}
Using the above Notation, the functors $V$ and $T$ induce a $k$-linear equivalence of categories
$$
\QGr A \equiv \QGr A^{(d)}
$$
for any $d\geq 1$.
\end{Thm}

Let $Q$ be a quiver with incidence matrix $M_Q$. For any natural number $d>0$, let $Q^{(d)}$ be the quiver associated with the $d$-th power $M_Q^d$. Since the $(v,u)$ entry of $M_Q^d$ is the number of paths from $v$ to $u$ of length $d$, it is easy to see that $k(Q^{(d)})$ is isomorphic, as a graded algebra, to the $d$-th Veronese subalgebra $kQ^{(d)}$ of $kQ$. Paul Smith observed this in \cite{Sm1} and uses A.B. Verevkin's result above to determine $\QGr kQ\equiv \QGr k(Q^{(d)})$.

\begin{Lem} \label{lem.pointtopoint}
The functor $V:\Gr kQ \to \Gr kQ^{(d)}$ sends the cyclic point module $\cO_v\in \Gr kQ$ to the cyclic point module based at $v$ in $\Gr kQ^{(d)}$. Also, $V(e_ikQ)=e_ikQ^{(d)}$. 
\end{Lem}
%\begin{proof}
%It can be seen that $V(\cO_v)$ is a point module. Let $p=(v_0=v,a_1,\ldots,v_{n-1},a_n,v_n=v)$ be the simple cycle which starts at $v$ and write $d=mn+i$ with $0\leq i<n$. for $j\geq 0$, write $jd=m'n'+i'$ with $0\leq i'<n$, then the space $V(\cO_v)_j=(\cO_v)_{jd}=kp^{m'}a_1\cdots a_{i'}$. The only path of length $d$ which does not annihilate $V(\cO_v)_j=(\cO_v)_{jd}$
%\end{proof}

\subsection{Extensions between point modules.} \label{sec.extofpoints}

The induced equivalence $V:\QGr kQ \to \QGr kQ^{(d)}$ gives an isomorphism
$$
\Ext^1_{\QGr kQ}(\cO_v,\cO_w)\cong \Ext^1_{\QGr kQ^{(d)}}(V(\cO_v),V(\cO_w)).
$$
So to determine when $\Ext^1_{\QGr kQ}(\cO_v,\cO_w)\neq 0$ we may first move the question to a suitable Veronese sub-algebra. 

\begin{proposition} \label{prop.nopath=noext}
Let $v$ and $w$ be cyclic vertices in a quiver $Q$ with finite GK-dimension. Let $n$ and $m$ be the lengths of the simple cycles which contain $v$ and $w$ respectively. If there are no paths from $v$ to $w$ whose length is a multiple of $nm$, then 
$$
\Ext^1_{\QGr kQ}(\cO_v,\cO_w)=0.
$$
\end{proposition}
\begin{proof}
If there are no paths in $Q$ from $v$ to $w$ of length $lnm$ for any $l\in \NN$, then there are no paths from $v$ to $w$ in $Q^{(nm)}$. So by moving to $kQ^{(nm)}$ using the Veronese equivalence, we may assume there are no paths from $v$ to $w$ in $Q$. Moreover, the only path from $v$ to itself of length $nm$ is $p^m$ where $p$ is the simple cycle based at $v$. Hence, in the quiver $Q^{(nm)}$, the simple cycle based at $v$ is a loop. Similar statements apply to the vertex $w$. Thus, we can assume there are no paths from $v$ to $w$ and the simple cycles which contain $v$ and $w$ are loops. Let $\{p,b_1,\ldots,b_l\}$ be all the arrows with source $v$ where $p$ is the loop and let $q$ be the loop based at $w$. 

Since $\Ext^1_{\QGr kQ}(\cO_v,\cO_w)$ is a quotient of $\Hom_{\QGr kQ}(\bigoplus p^mb_ikQ,\cO_w)$, we just need to show the latter is zero. However, since
$$
\Hom_{\QGr kQ}(\bigoplus p^mb_ikQ,\cO_w)\cong \prod \Hom_{\QGr kQ}(p^mb_ikQ,\cO_w),
$$
we just need to show $\Hom_{\QGr kQ}(p^mb_ikQ,\cO_w)=0$ for all $(m,i)$.

Since there are no paths from $v$ to $w$, every element of $p^mb_ikQ$ is annihilated by $e_w$. However, every nonzero element of $\cO_w$ is not annihilated by $e_w$ so we see there can be no nonzero graded homomorphisms from any submodule of $p^mb_ikQ$ to $\cO_w$. Hence, if $M\subset p^mb_ikQ$, then $\Hom_{\Gr kQ}(M,\cO_w)=0$. Therefore,
$$
\Hom_{\QGr kQ}(p^mb_ikQ,\cO_w)=\varinjlim \Hom_{\Gr kQ}(M,\cO_w)=0
$$    
where the direct limit is over all $M$ such that $p^mb_ikQ/M$ is torsion.
\end{proof}    

\begin{corollary} \label{cor.nopath=noext}
Let $kQ$ be a path algebra of finite GK-dimension. If $v$ and $w$ are two distinct cyclic vertices in the same cycle, then
$$
\Ext^1_{\QGr kQ}(\cO_v,\cO_w)=0.
$$
\end{corollary}
\begin{proof}
Let $p=(v=v_0,a_1,\ldots,w=v_{i+1},a_i\ldots,a_n,v_n=v)$ be the simple cycle which contains $v$ and $w$. Since the only paths in $Q$ from $v$ to $w$ have the form $p^la_1\cdots a_i$, all the paths have length $l\cdot n+i$ with $1\leq i<n$. Since $l\cdot n+i$ is not a multiple of $n$, Proposition \ref{prop.nopath=noext} implies $\Ext^1(\cO_v,\cO_w)=0$. 
\end{proof}

\begin{proposition} \label{prop.noselfext}
Let $v$ be a cyclic vertex in a path algebra $kQ$ of finite non-zero GK-dimension. Then
$$
\Ext^1_{\QGr kQ}(\cO_v,\cO_v)=0.
$$
\end{proposition}
\begin{proof}
If $n$ is the length of the simple cycle $p$ which contains $v$, then the only path of length $ln$ from $v$ to itself is $p^l$ $(p^0=e_v)$. By taking the $n$-th Veronese of $kQ$, we may assume the simple cycle which contains $v$ has length $1$, $p=(v,a,v)$. Hence, we have the situation
$$
\UseComputerModernTips
\xymatrix{ {v}\ar@<2ex>[r]^{b_1} \ar@<-2ex>[r]_{b_l}^{\vdots} \ar@(ul,dl)[]_{p} & {} }
$$ 
where the arrows $b_1,\ldots,b_l$ are all the arrows with source $v$ and $p$ is the loop at $v$. Using the exact sequence \ref{seq.extcomp} we know $\Ext^1(\cO_v,\cO_v)$ is a quotient of 
$$
\Hom(\bigoplus \pi^*p^mb_ikQ,\cO_v)\cong \prod \Hom_{\QGr kQ}(\pi^* p^mb_ikQ,\cO_v).
$$ 
Since every element of $p^mb_ikQ$ is annihilated by $e_v$ but every nonzero element of $\cO_v$ is not, there are no nonzero morphisms from any submodule of $p^mb_ikQ$ to $\cO_v$. Hence
$$
\Hom_{\QGr kQ}(\pi^*p^mb_ikQ,\cO_v)=\varinjlim \Hom_{\Gr kQ}(M,\cO_v)=0
$$ 
where the limit is over all $M\subset p^mb_ikQ$ whose cokernel is torsion. As this holds for all $(m,i)$, we get $\Hom_{\QGr kQ}(\bigoplus \pi^*p^mb_ikQ,\cO_v)=0$ which shows
$$
\Ext^1_{\QGr kQ}(\cO_v,\cO_v)=0.
$$
\end{proof}

\begin{Thm} \label{thm.extofpoints}
Let $kQ$ be a path algebra of finite non-zero GK-dimension. Let $v$ and $w$ be cyclic vertices which are in cycles of length $n$ and $m$ respectively. Then
$$
\Ext^1_{\QGr kQ}(\cO_v,\cO_w)\neq 0
$$
if and only if $v\neq w$ and there is a path from $v$ to $w$ of length $lnm$ for some $l\in \NN$. 
\end{Thm}
\begin{proof}
See section \ref{sec.proofof8.5}
\end{proof}

Recall the vertices of the Ext-quiver are the cyclic vertices of $Q$ and there is an arrow $v\to w$ in $E_Q$ if and only if there is a path in $Q$ from $v$ to $w$ of length $lmn$ where $n$ (respectively $m$) is the length of the simple cycle that contains $v$ (respectively $w$). Hence, Theorem \ref{thm.extofpoints} can be rephrased to say that $\Ext^1_{\QGr kQ}(\cO_v,\cO_w)\neq 0$ if and only if there is an arrow from $v$ to $w$ in $E_Q$. This is precisely what Theorem \ref{thm.main2Q} says. Theorems \ref{thm.simpleispoint} and \ref{thm.extofpoints} show the Ext-quiver $E_Q$ is determined by the simple objects of $\QGr kQ$ and their extensions. Since equivalences preserve simple objects and extensions, it follows that $\QGr kQ \equiv \QGr kQ'$ implies $E_Q=E_{Q'}$. This establishes one direction of Theorem \ref{thm.main3Q}.

\subsubsection{Proof of Theorem \ref{thm.extofpoints}.} \label{sec.proofof8.5}
If $v=w$ or if there are no paths from $v$ to $w$ of length $lmn$ for some $l>0$ then $\Ext^1(\cO_v,\cO_w)=0$ by Propositions \ref{prop.nopath=noext} and \ref{prop.noselfext}.

Suppose $v\neq w$ and there is a path of length $lmn$ for some $l>0$. Let $r:v\to w$ be a path of length $lmn$ for some $l>0$. If we take the $lmn$-th Veronese subalgebra $kQ^{(lmn)}$, then $r$ is a path of length $1$ in the quiver $Q^{(lmn)}$ and the simple cycles based at $v$ and $w$ become loops. Due to the Veronese equivalence $\QGr kQ \equiv \QGr kQ^{(lmn)}$ we can assume there is an arrow $r:v\to w$ in $Q$ and the simple cycles which contain $v$ and $w$ are just loops. Let $p$ be the loop at $v$ and $q$ the loop at $w$.

Consider the graded representation $M=(M_u,M_a)$ where $M_v=k[t]$, with its usual grading, and $M_u=0$ for all $u\neq v$. The map $M_p:M_v\to M_v$ is multiplication by $t$ and all other arrows act trivially. The graded module that $M$ determines is a point module which is seen to be the cyclic point module $\cO_v$. There is a similar description of $\cO_w$ as a graded representation.

Let $\nu\in k^{\NN}$ and define $N=N(\nu)=(N_u,N_a)$ to be the graded representation determined by the following data:
\begin{itemize}
\item As graded vector spaces, $N_v=N_w=k[t]$ where $\deg(t)=1$,
\item $N_u=0$ for all other vertices $u$,
\item For the two loops $p$ and $q$, $N_p$ and $N_q$ are just multiplication by $t$, 
\item For the arrow $r:v \to w$, $N_r:k[t] \to k[t]$ is given by
$$
N_r(t^i)=\nu_it^{i+1}.
$$
\item All other arrows act trivially.
\end{itemize}

As a graded vector space $N(\nu)=k[t]\oplus k[t]$ where all of the trivial paths except $e_v$ and $e_w$ and all arrows other than $p,q,r$ act trivially. The trivial paths $e_v$ and $e_w$ and the arrows $p,q$ and $r$ act as described above, e.g, $(\alpha t^i,\beta t^i)e_v=(\alpha t^i,0)$, $(\alpha t^i,\beta t^i)r=(0,\alpha \nu_i t^{i+1})$, etc $\ldots$ 

For each $\nu$, the point module $\cO_w$ sits inside $N(\nu)$ as the submodule $(0,k[t])$. Moreover, the quotient $N(\nu)/\cO_w$ is just $\cO_v$. Hence, for each $\nu \in k^{\NN}$, there is an exact sequence
$$
\UseComputerModernTips
\xymatrix{ {0}\ar[r] & {\cO_w}\ar[r] & {N(\nu)}\ar[r] & {\cO_v}\ar[r] & {0} }
$$ 
in $\Gr kQ$ which is also exact when considered in $\QGr kQ$.

As before, let $\{p,b_1=r,\ldots,b_l\}$ be all arrows whose source is $v$ where $p$ is the loop at $v$ and $b_1=r$. Using the projective resolution for $\cO_v$ and the exact sequence involving $N(\nu)$, we can construct the following commutative diagram
$$
\UseComputerModernTips
\xymatrix{ {0}\ar[r] & {\oplus p^mb_ikQ}\ar[r] \ar[d]_{\alpha_1} & {e_vkQ}\ar[r] \ar[d]^{\alpha_0} & {\cO_v}\ar[r] \ar[d]^{=} & {0} \\
           {0}\ar[r] & {\cO_w}\ar[r] & {N(\nu)}\ar[r] & {\cO_v}\ar[r] & {0} }
$$
where the maps $\alpha_0=\alpha_0(\nu)$ and $\alpha_1=\alpha_1(\nu)$ exist by the projectivity of $e_vkQ$ and $\oplus p^mb_ikQ$.

The map $\alpha_0$ is completely determined by $\alpha_0(e_v)$. If $\alpha_0(e_v)=(a,b)\in N(\nu)_{0}$, then $(a,b)=\alpha_0(e_v)=\alpha_0(e_v)e_v=(a,b)e_v=(a,0)$ shows $b=0$. Also, since the right square in the diagram above must commute we determine $a=1$, that is, $\alpha_0(e_v)=(1,0)$.  

Any non-trivial path which starts at $v$ has the form $p^m$ or $p^mb_ip'$ where $p'$ is any path whose source is $t(b_i)$. Since all arrows other than $p,q$ and $r$ annihilate $N(\nu)$, the only paths starting at $v$ which are not necessarily in the kernel of $\alpha_0$ are those of the form $p^i$ and $p^irq^j$ for $i,j\geq 0$. However, for paths of the form $p^i$ and $p^irq^j$;
\begin{eqnarray*}
\alpha_0(p^i)&=&\alpha_0(e_v)p^i=(1,0)p^i=(t^i,0), \\
\alpha_0(p^irq^j)&=&(1,0)p^irq^j=(t^i,0)rq^n=(0,\nu_it^{i+1})q^j=(0,\nu_it^{i+j+1}).
\end{eqnarray*} 

The map $\alpha_1$ is just the restriction of $\alpha_0$ to the submodule $\oplus p^mb_ikQ$. It is completely determined by what happens on the paths $p^mb_i$;
\begin{itemize}
\item $\alpha_1(p^mb_i)=0$ for $i>1$, 
\item $\alpha_1(p^mb_1)=\alpha_1(p^mr)=\nu_mt^{m+1}$.
\end{itemize}

The map $\alpha_1$ determines an element in $\Ext^1_{\QGr kQ}(\cO_v,\cO_w)$(computed using projective resolutions). Hence, we have a map 
$$
\Phi: k^{\NN} \to \Ext^1_{\Gr kQ}(\cO_v,\cO_w)
$$ 
which sends the element $\nu$ to $\alpha_1(\nu)$. If $a$ is any scalar in $k$ and $\nu \in k^{\NN}$, then the map $\alpha_1(a\nu)$ is the map determined by $\alpha_1(a\nu)(p^mr)=a\nu_mt^{m+1}$. Hence, $\alpha_1(a\nu)=a\alpha_1(\nu)$. Also, if $\nu,\mu \in k^{\NN}$, then $\alpha_1(\nu+\mu)(p^mr)=(\nu_m+\mu_m)t^{m+1}=(\alpha_1(\nu)+\alpha_1(\mu))(p^mr)$. Hence, the map $\Phi$ is a linear map.

Since the $k$-linear exact functor $\pi^*$ induces a map of vector spaces 
$$
\Ext^1_{\Gr kQ}(\cO_v,\cO_w) \to \Ext^1_{\QGr kQ}(\cO_v,\cO_w),
$$ 
we have a linear map from $\Phi^*:k^{\NN} \to \Ext^1_{\QGr kQ}(\cO_v,\cO_w)$. It will be shown that the kernel of $\Phi^*$ is all infinite sequences which are eventually zero.

Notice that for each $\nu$, the graded module $N(\nu)$ is torsion free and finitely generated. Since $N(\nu)$ is finitely generated, if $N'$ is a graded submodule of $N(\nu)$ such that $N(\nu)/N'$ is torsion, then $N(\nu)/N'$ must be finite dimensional. Hence, $N'$ must contain $N(\nu)_{\geq n}$ for some $n\in \NN$.

Let $\nu$ and $\mu$ be elements of $k^{\NN}$. Since every submodule $N'$ of $N(\nu)$ which has a torsion cokernel must contain $N(\nu)_{\geq n}$ for some $n\in \NN$, and $N(\mu)$ is torsion free, it follows that every morphism $f:\pi^*N(\nu) \to \pi^*N(\mu)$ is represented by a graded module morphism $\phi:N(\nu)_{\geq n}\to N(\mu)$ for some $n$.

Let $\phi:N(\nu)_{\geq n} \to N(\mu)$ be a morphism of graded modules. Then $\phi$ is determined by $\phi_{v}:k[t]_{\geq n} \to k[t]$ and $\phi_{w}:k[t]_{\geq n}\to k[t]$ such that 
\begin{itemize}
\item $\phi_v \circ N(\nu)_p=N(\mu)_p \circ \phi_v$,
\item $\phi_w\circ N(\nu)_q=N(\mu)_q \circ \phi_w$,
\item $N(\mu)_r\circ \phi_v=\phi_w\circ N(\nu)_r.$
\end{itemize}

As $N(\nu)_p$, $N(\nu)_q$, $N(\mu)_p$ and $N(\mu)_q$ are all just multiplication by $t$, the first two bullets above just say $\phi_v$ and $\phi_w$ are graded $k[t]$-module homomorphisms. As every graded $k[t]$-module homomorphism from $k[t]_{\geq n}\to k[t]$ is just multiplication by a scalar, we get a pair $(\alpha,\beta)\in k\times k$ such that $\phi_v(t^i)=\alpha t^i$ and $\phi_w(t^i)=\beta t^i$ for all $i\geq n$.

The last bullet point above indicates that for all $i\geq n$,
$$
\alpha \mu_it^{i+1}=N(\mu)_r\circ \phi_v(t^i)=\phi_w\circ N(\nu)_r(t^i)=\beta \nu_i t^{i+1}.
$$
Hence, $\alpha \mu_i=\beta \nu_i$ for all $i\geq n$.

Conversely, if there is a pair $(\alpha,\beta)$ such that $\alpha \mu_i=\beta \nu_i$ for all $i\geq n$, then we have a graded module morphism $\phi:N(\nu)_{\geq n} \to N(\mu)$ and hence a morphism $f:\pi^*N(\nu) \to \pi^*N(\mu)$.

Suppose $f:\pi^*N(\nu)\to \pi^*N(\mu)$ and $g:\pi^*N(\mu) \to \pi^*N(\lambda)$ are represented by pairs $(\alpha,\beta)$ and $(\gamma,\delta)$ respectively. Then the composite $g\circ f:\pi^*N(\nu) \to \pi^*N(\lambda)$ is represented by the pair $(\gamma \alpha, \delta \beta)$. Also, if $f,g:\pi^*N(\nu) \to \pi^*N(\mu)$ are represented by the pairs $(\alpha,\beta)$ and $(\gamma,\delta)$, then $f+g$ is represented by $(\alpha+\gamma,\beta+\delta)$. Hence, there is an isomorphism
$$
\Hom_{\QGr kQ}(\pi^*N(\nu),\pi^*N(\mu)) = \{(\alpha,\beta)\in k\times k\;|\; \alpha \mu_i=\beta \nu_i \text{ for }i\gg 0\}
$$
under which composition of maps is given by multiplication of pairs.

Suppose $\nu$ is a sequence such that for all $n\in \NN$, there exists a $m\geq n$ such that $\nu_{m}\neq 0$. Let $(\alpha,\beta)$ represent a morphism $f:\pi^*N(\nu) \to \pi^*N(\nu)$. Since $\beta \nu_i=\alpha \nu_i$ for all $i\gg 0$, and since there are arbitrarily large $i$ for which $\nu_i\neq 0$, we get $\alpha=\beta$. Hence, there is an isomorphism
$$
\End_{\QGr kQ}(\pi^*N(\nu)) \cong k
$$
which respects multiplication.

Again, let $\nu$ be an infinite sequence which is not eventually the zero sequence. Since $\End(\pi^*N(\nu))=k$ has no nontrivial idempotents, the object $\pi^*N(\nu)$ is indecomposable. Thus, $\pi^*N(\nu)$ is a nontrivial extension of $\cO_w$ by $\cO_v$. This shows that any sequence $\nu$ which is not eventually the zero sequence is not in the kernel of $\Phi^*$. Therefore, 
$$
\Ext^1_{\QGr kQ}(\cO_v,\cO_w)\neq 0
$$
and Theorem \ref{thm.extofpoints} is proved. 

Suppose $\nu$ is a sequence which is eventually zero. Find $n$ such that $\nu_i=0$ for all $i\geq n$. Consider the graded subspace $\{(at^i,0)\in N(\nu)\;|\; a\in k, i\geq n\}$. Since $(at^i,0)r=(0,0)$, this is a graded submodule of $N(\nu)$ which is isomorphic to $(\cO_v)_{\geq n}$. Hence, we have a map $\phi:(\cO_v)_{\geq n} \to N(\nu)$ such that the composition 
$$
\UseComputerModernTips
\xymatrix{ {(\cO_v)_{\geq n}}\ar[r]^{\phi} & {N(\nu)}\ar[r] & {\cO_v} }
$$ 
is the inclusion. Thus, the map $\phi:(\cO_v)_{\geq n} \to N(\nu)$ determines a map $\cO_v \to \pi^*N(\nu)$ such that the composition
$$
\UseComputerModernTips
\xymatrix{ {\cO_v}\ar[r] & {\pi^*N(\nu)}\ar[r] & {\cO_v} }
$$
is the identity. Therefore, the exact sequence
$$
\UseComputerModernTips
\xymatrix{ {0}\ar[r] & {\cO_w}\ar[r] & {\pi^*N(\nu)}\ar[r] & {\cO_v}\ar[r] & {0} }
$$
splits which shows the map $k^{\NN}\to \Ext^1_{\QGr kQ}(\cO_v,\cO_w)$ sends $\nu$ to $0$. Hence, the kernel of $\Phi^*$ is the subspace $\Fin$ of all infinite sequences which are eventually zero and we get a vector space embedding
$$
k^{\NN}/\Fin \hookrightarrow \Ext^1_{\QGr kQ}(\cO_v,\cO_w).
$$

\section{Properties of the Ext-quiver $E_Q$.} \label{sec.extposet}

As before, $kQ$ is assumed to have finite GK-dimension. Recall the set $\mC(Q)$, which consists of all the simple cycles in $Q$, is a poset. The relation being $C_1\preceq C_2$, for simple cycles $C_1$ and $C_2$, if there is a chain from a vertex of $C_1$ to a vertex $C_2$. The Growth of $kQ$ is then polynomial of degree $d$ where $d$ is the maximal size of a totally ordered subset of $\mC(Q)$.  

It was mentioned in the introduction that the Ext-quiver $E_Q$ can be thought of as a poset whose elements are the cyclic vertices of $Q$ and where $u\preceq v$ if either $u=v$ or there is an arrow $u\to v$ in $E_Q$. Equivalently, we can define $u\prec v$ if $\Ext^1_{\QGr kQ}(\cO_u,\cO_v)\neq 0$. The next proposition justifies the claim that $E_Q$ is a poset.

\begin{proposition} \label{prop.simpleposet}
$(E_Q,\preceq)$ is a poset.
\end{proposition}
\begin{proof}
By definition $v\preceq v$. Suppose $v\preceq w$ and $w\preceq v$. If $v\neq w$, the vector spaces $\Ext^1(\cO_v,\cO_w)$ and $\Ext^1(\cO_w,\cO_v)$ are both non-zero. By corollary \ref{cor.nopath=noext} we know $v$ and $w$ cannot be in the same simple cycle. However, by Theorem \ref{thm.extofpoints} we know there must be a path from $v$ to $w$ and a path from $w$ to $v$. This is a contradiction however as this implies the existence of a cycle which contains both $v$ and $w$. Hence, $v=w$ must hold.

Suppose $u \preceq v$ and $v \preceq w$. If either $u=v$, $v=w$ or $u=w$ then $u\preceq w$. Suppose $u,v,w$ are all distinct cyclic vertices. Necessarily, all three vertices must be in distinct cycles. Let $n_u,n_v$ and $n_w$ be the lengths of the simple cycles, $p_u,p_v$ and $p_w$ which contain $u,v$ and $w$ respectively. Since $\Ext^1(\cO_u,\cO_v)$ and $\Ext^1(\cO_v,\cO_w)$ are not zero, there must be a path $q$ from $u$ to $v$ of length $ln_un_v$ for some $l\in \NN$ and a path $q'$ from $v$ to $w$ of length $l'n_vn_w$ for some $l'$. For any non-negative integers $\alpha$ and $\gamma$, $p_u^\alpha q$ is a path from $u$ to $v$ of length $ln_un_v+\alpha n_u$ and $q'p_w^{\gamma}$ is a path from $v$ to $w$ of length $l'n_vn_w+\gamma n_w$.

Pick $\beta \in \NN$ such that $\beta n_w\geq ln_v$ and let $\alpha=\beta n_w-ln_v$. Similarly, find $\delta \in \NN$ such that $\delta n_u \geq l'n_v$ and let $\gamma=\delta n_u-l'n_v$. Then $p_u^{\alpha}qq'p_w^{\gamma}$ is a path from $u$ to $w$ of length
$$
\alpha n_u+ln_un_v+l'n_vn_w+\gamma n_w=(\alpha+ln_v)n_u+(\gamma+l'n_v)n_w=(\beta+\delta)n_un_w.
$$ 
Hence, there is a path from $u$ to $w$ of length which is a multiple of $n_un_w$. Therefore, there is an arrow $u\to w$ in $E_Q$ which is to say $u\prec w$.
\end{proof} 

\begin{proposition} \label{prop.simposet}
Let $kQ$ be a path algebra of finite GK-dimension with $\mC(Q)$ and $E_Q$ the posets introduced previously. There is a totally ordered subset in $\mC(Q)$ of length $d$ if and only if there is a totally ordered subset of length $d$ in $E_Q$.
\end{proposition}
\begin{proof}
Let $v_1\prec \cdots \prec v_d$ be a chain in $E_Q$. Each of the cyclic vertices $v_i$ appear in a cycle $C_i$. Let $n_i$ be the length of the cycle $C_i$. Since $v_i\neq v_j$ but $v_i\preceq v_j$ for $i<j$ we know $\Ext^1(\cO_{v_i},\cO_{v_j})\neq 0$. In particular, the simple cycles $C_i$ and $C_j$ must be different for all pairs $(i,j)$. As $\Ext^1(\cO_{v_i},\cO_{v_{i+1}})\neq 0$, there must be a path from $v_i$ to $v_{i+1}$. Hence, there is a path from a vertex of $C_i$ to a vertex of $C_{i+1}$ so necessarily a chain from a vertex of $C_i$ to a vertex of $C_{i+1}$. Hence, $C_1\preceq \cdots \preceq C_d$ in $\mC(Q)$. 

Let $C_1\prec C_2$ be two simple cycles of lengths $n$ and $m$ respectively. Write $C_1=Q(p_1)$ where $p_1=(v_0,a_1,\ldots,a_n,v_n)$ and $C_2=Q(p_2)$ where $p_2=(w_0,b_1,\ldots,b_m,w_m)$. There is a chain $p$ from a vertex in $C_1$ to a vertex in $C_2$. We may assume the chain $p$ starts at $v_0$ and ends at $w_0$. Let $l$ be the length of $p$. Pick any cyclic vertex $v_i$ in $C_1$, then the path $a_i\cdots a_np$ is a path from $v_i$ to $w_0$ which has length $l+n-i$. Find $\alpha \in \NN$ such that $\alpha nm\geq l+n-i$ and let $\beta=\alpha nm-(l+n-i)$. Write $\beta=l'm+j$ with $j<\beta$. Then $a_i\cdots a_npp_2^{l'}b_1\cdots b_j$ is a path from $v_i$ to $w_j$ of length 
$$
l+n-i+l'm+j=l+n-i+\beta=\alpha nm.
$$
Hence, $\Ext^1(\cO_{v_i},\cO_{w_j})\neq 0$ and we have $v_i\prec w_j$.

Now suppose we have a chain $C_1\prec \cdots<C_{d-1}\prec C_d$ with $n_i$ the length of $C_i$. By induction, we can find vertices $v_i$ in $C_i$ for $i<d$ such that $v_1\prec \cdots \prec v_{d-1}$. Now consider the vertex $v_{d-1}$ of $C_{d-1}$. Using the same argument as above we can find a vertex $v_d$ in $C_d$ and a path from $v_{d-1}$ to $v_d$ of length $ln_{d-1}n_d$ for some $l$. Hence, $v_{d-1}\prec v_d$ and we have 
$$
v_1\prec \cdots \prec v_d
$$
in $E_Q$.
\end{proof}

\begin{corollary} \label{cor.GKinposet}
The GK-dimension of $kQ$ equals the maximal length of a totally ordered subset in $E_Q$.
\end{corollary}
\begin{proof}
This follows from the previous proposition since the GK-dimension of $kQ$ is the maximal length of a totally ordered subset in $\mC(Q)$. 
\end{proof}

\begin{Cor}
Suppose $Q$ and $Q'$ are quivers such that $\QGr kQ\equiv \QGr kQ'$. Then $\GKdim kQ=\GKdim kQ'$.
\end{Cor}
\begin{proof}
If $kQ$ has infinite GK-dimension it can be shown $\QGr kQ$ has infinitely many isomorphism classes of simple objects. Hence, $\QGr kQ'$ has infinitely many simple objects which implies $kQ'$ has infinite GK-dimension. Therefore, if one algebra has infinite GK-dimension then so does the other. 

If $kQ$ has finite GK-dimension then so must $kQ'$. Since $\QGr kQ\equiv \QGr kQ'$ implies $E_Q\cong E_{Q'}$, $\GKdim kQ=\GKdim kQ'$ by Corollary \ref{cor.GKinposet}.
\end{proof}

\section{The Grothendieck group of $\QGr kQ$.} \label{sec.grothgroup2} 
\subsection{}
The remarks following Theorem \ref{thm.extofpoints} established one direction of Theorem \ref{thm.main3Q}; that is $\QGr kQ \equiv \QGr kQ'$ implies $E_Q=E_{Q'}$. To establish the other direction, it will be shown that the Grothendieck group of the category $\QGr kQ$, computed using finitely generated projectives, is completely determined by the Ext-quiver $E_Q$. This is useful for the following reason:

In \cite{Sm1}, Smith showed that for any path algebra $kQ$ (no assumption on GK-dimension), there is an ultramatricial algebra $S(Q)$ and an equivalence of categories
$$
\QGr kQ \equiv \Mod S(Q)
$$ 
where $\Mod S(Q)$ is the category of right $S(Q)$-modules. The relationship between ultramatricial algebras and their Grothendieck groups is especially strong in the following sense:

\begin{Thm}[\cite{Good}, Cor. 15.27 page 222] \label{thm.goodcite}
Let $R$ and $S$ be ultramatricial algebras. Then $\Mod R \equiv \Mod S$ if and only if $\K0(R) \cong \K0(S)$ as pre-ordered abelian groups.
\end{Thm}

From the equivalence $\QGr kQ \equiv \Mod S(Q)$ and the fact that $S(Q)$ is ultramatricial, we get the following Corollary to Theorem \ref{thm.goodcite}

\begin{Cor} \label{cor.goodcite}
Let $kQ$ and $kQ'$ be path algebras. Then $\QGr kQ \equiv \QGr kQ'$ if and only if $\K0(\QGr kQ)\cong \K0(\QGr kQ')$ as pre-ordered abelian groups.
\end{Cor}

Hence, to prove that $E_Q=E_{Q'}$ implies $\QGr kQ \equiv \QGr kQ'$ for path algebras of finite GK-dimension, it will be shown that $E_Q=E_{Q'}$ implies $\K0(\QGr kQ) \cong \K0(\QGr kQ')$ as preordered abelian groups. 

One of the tools that will be useful in calculating $K_0(\QGr kQ)$ will be the equivalence $\QGr kQ \equiv \Mod S(Q)$. Since the algebra $S(Q)$ is a direct limit of matricial algebras, $K_0(S(Q))$ is computed using a direct limit of finite rank free groups. We start of with a description of the ultramatricial algebras $S(Q)$. 

\subsection{A description of $S(Q)$.}

Let $Q$ be any quiver. Assume we have labeled the vertices as $Q_0=\{1,\ldots,r\}$ and let $M_Q=(a_{ij})$ be the incidence matrix for $Q$ with respect to this labeling. The convention is that $a_{ij}$ is the number of arrows from vertex $i$ to vertex $j$. Consider the sequence of vectors
$$
\bp_{n}=\bmat
p_{n,1} \\
\vdots \\
p_{n,r}
\emat
$$ 
where $\bp_0=(1\cdots 1)^T$ and 
$$
\bp_{n+1}=M_Q^T\bp_{n}.
$$
Associated to this sequence of integer valued vectors is a sequence of matricial algebras
$$
S_n=\prod_{i=1}^{r}M_{p_{n,i}}(k).
$$
As $\bp_{n+1}=M_Q^T\bp_n$, we have unital algebra homomorphisms $\theta_n:S_n\to S_{n+1}$ determined by the matrix $M_Q^T$ (see Chapter 4 of \cite{Dav}). Let 
$$
S(Q)=\varinjlim_{n} S_n,
$$ 
$S(Q)$ is called ultrmatricial because it is a limit of matricial algebras. In \cite{Sm1}, Smith proves there is an equivalence of categories 
$$
\QGr kQ \equiv \Mod S(Q).
$$

Since $K_0$ commutes with direct limits, 
$$
\K0(\QGr kQ)\cong K_0(S(Q))\cong \varinjlim _{n}K_0(S_n)
$$
as preordered abelian groups. As $S_n\cong \prod_{i=1}^r M_{p_{n,i}}(k)$ is a matricial algebra, $K_0(S_n)\cong \ZZ^r$. Moreover, under the identifications $K_0(S_n)=K_0(S_{n+1})=\ZZ^r$, the morphism $K_0(\theta_n):K_0(S_n) \to K_0(S_{n+1})$ is given by left multiplication by $M_Q^T$ (see Chapter 4 of \cite{Dav}). 

Using the simple fact that $(M_Q^T\bv)^T=\bv^TM_Q$, if we view vectors in $\ZZ^r$ as row vectors, the map $K_0(\theta_n)$ is given by right multiplication by the incidence matrix $M_Q$. Viewing $\ZZ^r$ as row vectors and $\K0(\theta_n)$ as right multiplication by $M_Q$ will make some labeling choices later a little more convenient. 

Thus, to compute the group $K_0(S(Q))\cong K_0(\QGr kQ)$, we need to compute the direct limit
\begin{equation} \label{eqn.k0comp}
\UseComputerModernTips
\xymatrix{ {\ZZ^r}\ar[r]^{M_Q} & {\ZZ^r}\ar[r]^{M_Q} & {\cdots} \ar[r]^{M_Q} & {\ZZ^r}\ar[r]^{M_Q} & {\cdots} }
\end{equation}
Also, Since $K_0(S_n)^+=\NN^r$, the positive cone $K_0(\QGr kQ)^+$ is the union in $K_0(\QGr kQ)$ given by
\begin{equation} \label{eqn.k+comp}
\UseComputerModernTips
\xymatrix{ {\NN^r}\ar[r]^{M_Q} & {\NN^r}\ar[r]^{M_Q} & {\cdots} \ar[r]^{M_Q} & {\NN^r}\ar[r]^{M_Q} & {\cdots} }
\end{equation}

\begin{Prop} \label{prop.k0comp}
Let $Q$ be a quiver. The Grothendieck group $(K_0(\QGr kQ))$ is isomorphic to the direct limit of the system \ref{eqn.k0comp} where the elements of $\ZZ^r$ are considered as row vectors and the maps are given by right multiplication by the incidence matrix $M_Q$. Moreover, the positive cone is the subset given by the direct limit of \ref{eqn.k+comp}. 
\end{Prop}

\begin{example}
Let $Q$ be the quiver
$$
\UseComputerModernTips
\xymatrix{ {\bullet}\ar[r] \ar@(ul,dl)[] & {\bullet} \ar@(ur,dr)[] }
$$
with
$$
M_Q=\bmat
1 & 1 \\
0 & 1
\emat.
$$

A simple induction shows
$$
M_Q^n=\bmat
1 & n \\
0 & 1
\emat.
$$
As $M_Q$ is invertible over $\ZZ$, the direct limit of \ref{eqn.k0comp} is $\ZZ^2$ with maps
$$
\UseComputerModernTips
\xymatrix{ {\ZZ^2}\ar[r]^{M_Q}\ar[drr]_{I} & {\ZZ^2}\ar[r]^{M_Q} \ar[dr]^{M_Q^{-1}} & {\cdots} \ar[r]^{M_Q} & {\ZZ^2}\ar[r]^{M_Q} \ar[dl]^{M_Q^{-n+1}} & {\cdots} \\
{} & {} & {\ZZ^2} & {} & {} } 
$$
Hence, $K_0(\QGr kQ)\cong \ZZ^2$. The positive cone is all elements $(z_1,z_2)\in \ZZ^2$ such that $(z_1,z_2)M_Q^n \in \NN^2$ for $n\gg 0$. As
$$
(z_1,z_2)M_Q^n=(z_1,nz_1+z_2)
$$
we determine $(z_1,z_2)\in K_0^+(\QGr kQ)$ if and only if $(z_1,z_2)\in (0,\NN)$ or $(z_1,z_2)\in (\NN_{>0},\ZZ)$. 
\end{example}

\subsection{An ordered abelian group associated to a finite poset.} 

Let $(\cP,\preceq)$ be a finite poset. Without loss of generality, we can assume $\cP=\{1,2,\ldots,n\}$ and $i\preceq j$ implies $i\leq j$ where $\leq$ is the usual ordering of integers. Associate to $\cP$ the free abelian group $\ZZ^n$ of rank $n=|\cP|$ and let $e_i$ be the element $(0,\cdots,1,\cdots,0)$ where the $1$ is in the $i$-th position. For each integer $i\in \cP$, define 
$$
\Delta(\cP)_i=\left(\NN^+e_i+\sum_{i \prec k}\ZZ e_k\right) \cup \{0\}
$$ 
which is a submonoid of $\ZZ^n$ and let $\Delta(\cP)$ be the submonoid $\sum_{i=1}^{n} \Delta(\cP)_i$.

\begin{Lem}
The monoid $\Delta(\cP)$ generates $\ZZ^n$ as an abelian group and is a strict positive cone.
\end{Lem}
\begin{proof}
That $\Delta(\cP)$ generates $\ZZ^n$ as an abelian group follows since $e_i\in \Delta$ for all $i\in \cP$. Every element $\bv \in \Delta(\cP)$ can be written as
$$
\bv=\sum_{i\in \cP}(n_ie_i+\sum_{i\prec k}{z_k^i}e_{k})
$$
for some nonnegative integers $n_i$ and integers $z^i_k$. If $n_i=0$ then $z_k^i=0$ for all $k$. Let $m$ be minimal, in the usual $\leq$ ordering on integers, such that $n_m\neq 0$. Then 
$$
\bv=(0,\ldots,0,n_m,\ldots)
$$
where $n_m>0$. Hence, if $\bv \in \Delta(\cP)$ is not zero, then the first nonzero entry in $\bv$ is positive. Hence, the first nonzero entry of $-\bv$, for $\bv\in \Delta(\cP)$, is negative and so cannot be in $\Delta(\cP)$. Thus, the only element $\bv$ such that $\bv$ and $-\bv$ are in $\Delta(\cP)$ is $\bv=0$.
\end{proof} 

Let $\bv=(v_1,\ldots,v_n)\in \ZZ^n$. Define $\supp(\bv)$, the support of $\bv$, to be
$$
\supp(\bv)=\{1\leq i\leq n\;|\; v_i\neq 0\}.
$$
Since $\supp(\bv)$ is a subset of $\cP$, it inherits a poset structure whenever it is not empty.

\begin{Lem} \label{lem.altdelta}
Let $\bv=(v_1,\ldots,v_n)$ be a nonzero element in $\ZZ^n$. Then $\bv \in \Delta(\cP)$ if and only if $v_j>0$ for all $j$ minimal in $\supp(\bv)$. 
\end{Lem}
\begin{proof}
Suppose $\bv=(v_1,\ldots,v_n)=\sum_{i\in P}(n_ie_i+\sum_{i\prec k}{z_k^i}e_{k}) \in \Delta(\cP)$. If we write
$$
n_ie_i+\sum_{i\prec k}{z_k^i}e_{k}=(z_1^i,\ldots,z_n^i),
$$
then $z_k^i=0$ if $i\not\preceq k$ and $z_i^i=n_i\geq 0$. Also, $n_i=0$ implies $z_k^i=0$ for all $k$.

Using this notation, 
$$
v_j=\sum_{i=1}^{n}{z_j^i}=n_j+\sum_{i\neq j}z_j^i.
$$ 

$(\Rightarrow)$Suppose $v_j<0$. Since $n_j \geq 0$, this forces $z_{j}^{i_1}< 0$ for some $i_1\prec j$. As $z_j^{i_1}\neq 0$, then necessarily $n_{i_1} \neq 0$. If $v_{i_1}\neq 0$, then $i_1\in \supp(\bv)$ and $i_1\prec j$ which shows $j$ is not minimal in $\supp(\bv)$. If $v_{i_1}=0$, then since $v_{i_1}=n_{i_1}+\sum_{i\neq i_1}z_{i_1}^i$ and $n_{i_1}> 0$, there exists an $i_2\prec i_1$ such that $z_{i_1}^{i_2}<0$ and hence that $n_{i_2}>0$. If $v_{i_2}\neq 0$ then $i_2\prec j$ and $i_2\in \supp(\bv)$ which shows $j$ is not minimal. If $v_{i_2}=0$, we can continue this process to get a sequence
$$
j\succ i_1\succ i_2\cdots.
$$ 
However, as $j>i_1>i_2>\cdots >1$, this process must stop. The only way for it to stop is if an element $i_n$ is reached in which $v_{i_n}\neq 0$ since $v_{i_n}=0$ implies the existence of a $i_{n+1}$. As $v_{i_n}\neq 0$, $i_n\in \supp(\bv)$ and $i_n\prec j$ which shows $j$ is not minimal. Thus, if $j\in \supp(\bv)$ is minimal, then $v_j>0$.

$(\Leftarrow)$ Suppose $\bv=(v_1,\ldots,v_n)$ satisfies $v_j>0$ for all minimal $j\in \supp(\bv)$. Let $i_1<i_2<\cdots<i_l$ be all the minimal elements of $\supp(\bv)$. If $l=1$, then $\supp(\bv)$ has one minimal element and $v_k\neq 0$ implies $i_1\prec k$. Hence,
$$
\bv=v_{i_1}e_{i_1}+\sum_{i_1\prec k}v_ke_k \in \Delta(\cP).
$$

Now suppose $l>1$. Define
$$
\bv_1=v_{i_1}e_{i_1}+\sum_{i_1\prec k}v_ke_k.
$$
Since $v_{i_1}\geq 1$, we know $\bv_1\in \Delta(\cP)$. Let $\bv'=\bv-\bv_1=(v'_1,\ldots,v'_n)$. If $i_1\preceq k$, then $v'_k=0$ while $i_1\not \preceq k$ implies $v'_k=v_k$. Since $\supp(\bv')=\supp(\bv)\-\{k\;|\; i\preceq k\}$ and $i_2,\ldots i_l \in \supp(\bv')$, these are all of the minimal elements of $\supp(\bv')$. Therefore, $\bv'$ is a vector for which $v'_j>0$ for all minimal $j\in \supp(\bv')$ and there are only $l-1$ minimal elements. Hence, by induction, $\bv' \in \Delta(\cP)$ and it follows that $\bv=\bv_1+\bv' \in \Delta(\cP)$ since $\bv_1$ and $\bv'$ are in $\Delta(\cP)$.
\end{proof}

\subsection{$\K0(\QGr kQ)$ for a path algebra of finite GK-dimension.} \label{sec.K0proof}

Let $Q$ be a quiver such that $kQ$ has finite GK-dimension and let $M=M_Q$ be its incidence matrix. By Proposition \ref{prop.k0comp}, $K_0(\QGr kQ)$ is the direct limit of the direct system \ref{eqn.k0comp} with positive cone determined by \ref{eqn.k+comp}. Recall that we can view $E_Q$ as a finite poset.

\begin{Thm} \label{thm.k0qgrkQ}
Let $kQ$ be a path algebra of finite GK-dimension and $E_Q$ the associated Ext-quiver with $p=|(E_Q)_0|$ the number of cyclic vertices. There is an isomorphism 
$$
\phi:\K0(\QGr kQ) \cong \ZZ^p
$$
such that $\phi(\K0(\QGr kQ)^+)=\Delta(E_Q)$.
\end{Thm}

The proof of Theorem \ref{thm.k0qgrkQ} is somewhat long and will be given in the section \ref{sec.proofk0}. 

If $kQ$ and $kQ'$ are path algebras of finite GK-dimension such that $E_Q=E_Q'$, then it follows from Theorem \ref{thm.k0qgrkQ} that $\K0(\QGr kQ) \cong \K0(\QGr kQ')$ as pre-ordered abelian groups. Hence, by Corollary \ref{cor.goodcite} 
$$
\QGr kQ \equiv \QGr kQ'.
$$
Thus, the direction of Theorem \ref{thm.main3Q} which states that $E_Q=E_{Q'}$ implies $\QGr kQ \equiv \QGr kQ'$ has been established. These remarks together with the remarks following Theorem \ref{thm.extofpoints} proves Theorem \ref{thm.main3Q}.  

\subsubsection{Proof of Theorem \ref{thm.k0qgrkQ}.} \label{sec.proofk0}

Let $kQ$ be a path algebra of finite GK-dimension and let $\mC(Q)=\{\cC_1,\ldots,\cC_p\}$ be the set of simple cycles in $Q$. Let $l_i$ be the length of $\cC_i$ and $L=l_1\cdots l_p$. By Verevkin's result, the categories $\QGr kQ$ and $\QGr k(Q^{(L)})$ are equivalent. Moreover, by the choice of $L$, every cycle in $Q^{(L)}$ has length one, i.e, is just a loop at a cyclic vertex.  

Let $n$ be the number of vertices of $Q^{(L)}$ (equivalently $Q$). Since every closed path in $Q^{(L)}$ is just some power of a loop based at a cyclic vertex, we can label the vertices of $Q^{(L)}$ with the integers $\{1,\ldots,n\}$ in such a way that if there is an arrow from $i$ to $j$, then $i\leq j$. With this labeling, the incidence matrix $M^L=M_{Q^{(L)}}$ becomes upper triangular with the diagonal entries being either zero or one. The number of $1$'s down the diagonal is precisely the number of cyclic vertices in $Q^{(L)}$ (equivalently $Q$). Let $p$ be the number of cyclic vertices. 

Considering $M^L$ as a linear operator acting on the right on $\QQ^n$, it is possible that the rank of $M^L$ is greater than the number of ones down the main diagonal. However, we can put $M^L$ into Jordan canonical form over $\QQ$ as all eigenvalues of $M^L$ lie in $\QQ$. Hence, the Jordan form of $M^L$ is
$$
J(1,n_1)\oplus \cdots \oplus J(1,n_i)\oplus J(0,m_{1})\oplus \cdots \oplus J(0,m_q).
$$  
Here, $J(\lambda, m)$ is the Jordan block matrix of size $m$ and diagonal entries $\lambda$, $n_1+\cdots+n_i=p$ is the number of cyclic vertices and $n_1+\cdots+n_i+m_1+\cdots+m_q=n$. Since each matrix $J(0,m_j)$ is nilpotent of degree $m_j$ and $m_j\leq n$, $J(0,m_j)^n=0$. Hence, for all $m\geq n$, the matrix $(M^L)^m=M^{Lm}$ will have rank $n_1+\cdots+n_i=p$ which is the number of cyclic vertices in $Q$. 

Now $M^{Ln}$ is the incidence matrix for $(Q^{(L)})^{(n)}=Q^{(Ln)}$ under the given labeling of vertices. Also, since $n$ is the number of vertices of $Q^{(L)}$, if there is a path in $Q^{(L)}$ from a cyclic vertex $i$ to a cyclic vertex $j$, then there must be a path of length $n$. Hence, in the quiver $(Q^{(L)})^{(n)}=Q^{(Ln)}$, there is a path from $i$ to $j$ if and only if there is an arrow in $Q^{(Ln)}$ from $i$ to $j$ for cyclic vertices $i$ and $j$.  

Putting the last few paragraphs of discussion together, we now have the following situation. To compute $\K0(\QGr kQ)$, we assume $Q$ is a quiver such that:
\begin{itemize}
\item Every simple cycle of $Q$ is a loop at a cyclic vertex,
\item The incidence matrix $M$ is upper triangular with diagonal entries in $\{0,1\}$,
\item The rank of $M$ is $p$ which is the number of cyclic vertices,
\item For cyclic vertices $i$ and $j$, there is a path from $i$ to $j$ if and only if there is an arrow from $i$ to $j$.
\end{itemize}

For $m\in \NN$, let $a_{ij}(m)$ be the $ij$-th entry of $M^m$ and $R_i(m)$ the $i$-th row of $M^m$. Let $\nu_1<\cdots <\nu_p$ be the indices for which the diagonal elements $a_{\nu_i,\nu_i}=1$, i.e., $\{\nu_1,\ldots,\nu_p\}$ are the cyclic vertices of $Q$. With this labeling of the cyclic vertices, we realize the poset $E_Q$ as the set $\{1,\ldots,p\}$ and for which $i\prec j$ if $\nu_i\prec \nu_j$, that is, if there is an arrow in $Q$ from $\nu_i \to \nu_j$.

The rows $\{R_{\nu_1},\ldots,R_{\nu_p}\}$ provide a $\QQ$-basis for $\QQ^nM$ as $M$ has rank $p$ and the rows $R_{\nu_i}$ are linearly independent. Suppose $\bv \in \ZZ^n$, then there are rational numbers $b_i$ such that
$$
\bv M=\sum_{i=1}^pb_iR_{\nu_i}.
$$
Find $j$ minimal in the usual ordering of integers such that $b_j\neq 0$. Since the first nonzero entry of $R_{\nu_j}$ is a one and all the nonzero entries of $R_{\nu_l}$ for $l>j$ occur further to the right, the first nonzero entry of $\sum_{i=j}^pb_iR_{\nu_i}$ is $b_j$. Hence, $b_j$ must be an integer since $\bv M$ has integer entries. Considering the element 
$$
\bv M-b_jR_{\nu_j}=(\bv-b_je_{\nu_j})M=\sum_{i=j+1}^pb_iR_{\nu_i}
$$ 
we determine by the same reasoning that the next nonzero $b_l$ for $l>j$ is also an integer and can continue to conclude that each $b_i$ is an integer for all $i$. Therefore, the vectors $\{R_{\nu_1},\ldots,R_{\nu_p}\}$ provide a $\ZZ$-basis for the group $\ZZ^nM$.

For any row $R_{\nu_i}$ of $M$, we can write
$$
R_{\nu_i}M=\sum_{j=1}^{p}b_{ij}R_{\nu_j}
$$
for some integers $b_{ij}$. Let $N=(b_{ij})$ be the $p\times p$ matrix with $b_{ij}$ as the $ij$-th entry and $R$ be the $p\times n$ matrix whose $i$-th row is $R_{\nu_i}$: 
$$
R=\bmat
R_{\nu_1} \\
\vdots \\
R_{\nu_p}
\emat.
$$
By the definition of $N$ and $R$, 
$$
NR=RM.
$$

Note that $R_{\nu_i}M=R_{\nu_i}(2)$, the $\nu_i$-th row of $M^2$. As $M$ is upper triangular, we know $M^2$ is upper triangular, moreover, the $(\nu_i,\nu_i)$ entry of $M^m$ is a one for all $m$. Hence, the first nonzero entry of $R_{\nu_i}M$, which is a $1$, occurs in column $\nu_i$. As $R_{\nu_i}M=\sum b_{ij}R_{\nu_j}$ and $R_{\nu_j}$ has a $1$ in the $\nu_j$-th column, we get $b_{ij}=0$ for $j<i$. Also, from this we see $b_{ii}=1$ as $b_{ii}$ is the first nonzero entry of $\sum b_{ij}R_{\nu_j}$.

As the entries of $N$ are the $b_{ij}$, the fact that $b_{ij}=0$ if $j<i$ and $b_{ii}=1$ implies $N$ is an upper triangular matrix with $1$'s down the main diagonal. Hence, $N$ is an automorphism of $\ZZ^p$. Consider the following commutative diagram:
\begin{equation} \label{eqn.freek0}
\UseComputerModernTips
\xymatrix{ {\ZZ^p}\ar[r]^{N} \ar[d]_{R} & {\ZZ^p}\ar[r]^{N} \ar[d]_{R} & {\cdots}\ar[r]^{N} & {\ZZ^p}\ar[r]^{N} \ar[d]_{R} & {\ZZ^p}\ar[r]^{N} \ar[d]_{R} & {\cdots} \\
           {\ZZ^n}\ar[r]_{M} & {\ZZ^n}\ar[r]_{M} & {\cdots}\ar[r]_{M} & {\ZZ^n}\ar[r]_{M} & {\ZZ^n}\ar[r]_{M} & {\cdots} }
\end{equation}    
where the maps are right multiplication by the matrix listed. 

\begin{Prop} \label{prop.rankp}
The group $\K0(\QGr kQ)$ is free of rank $p$.
\end{Prop}
\begin{proof}
As $N$ is an automorphism, the direct limit of the top row in diagram \ref{eqn.freek0} is just $\ZZ^p$. The direct limit of the bottom row is the Grothendieck group $\K0(\QGr kQ)$. Since the rows of the matrix $R$ are linearly independent, right multiplication by $R$ is an injective map. Also, all the squares in the diagram commute so $R$ induces an injective map $\phi:\ZZ^p \to \K0(\QGr kQ)$.

Since the rows of $R$ are a basis for the image of the matrix $M$, we know $\ZZ^nM=\ZZ^pR$. Thus, the map $\phi$ is surjective and hence an isomorphism. This establishes the fact that $\K0(\QGr kQ)$ is a free abelian group of rank $p$ which is the number of cyclic vertices. 
\end{proof}

Because $N$ can have negative entries we cannot get the positive cone directly from the previous isomorphism induced by $R$. Nevertheless, $N$ is still useful in computing the positive cone.

If $\sM \in \QGr kQ$ is finitely generated projective, there is an isomorphism
$$
\sM \cong \bigoplus_{i\in Q_0}\left( \bigoplus_{j\in J_i} \pi^*e_ikQ(-n_{ij})^{m_{ij}}\right)
$$
for some finite sets $J_i$, $n_{ij}\in \ZZ$, and $m_{ij}\in \NN$(see Proposition 3.2 in \cite{Sm1}). Hence, the Grothendieck group $\K0(\QGr kQ)$ is generated by the elements $[\pi^*e_ikQ(z)]$ for $i\in Q_0$ and $z\in \ZZ$.

In $\Gr kQ$, 
$$
(e_ikQ)_{\geq 1}=\bigoplus_{a\in Q_1, s(a)=i}akQ.
$$
Since $akQ\cong e_{t(a)}kQ(-1)$, 
$$
(e_ikQ)_{\geq 1}\cong \bigoplus_{a\in Q_1,s(a)=i}e_{t(a)}kQ(-1)=\bigoplus_{j\in Q_0} e_jkQ(-1)^{\oplus a_{ij}}
$$
where $a_{ij}$ is the number of arrows from $i$ to $j$. Hence, in $\QGr kQ$, we get
$$
\pi^*e_ikQ \cong \pi^*((e_ikQ)_{\geq 1}) \cong \bigoplus_{j\in Q_0} \pi^*e_jkQ(-1)^{\oplus a_{ij}}
$$
and this implies,
$$
\pi^*e_ikQ(z) \cong \bigoplus_{j\in Q_0} \pi^*e_jkQ(z-1)^{\oplus a_{ij}}
$$
for $z\in \ZZ$.

Write $\sP_i(z)$ for $\pi^*e_ikQ(z)$. The above isomorphism gives the following relation in $\K0(\QGr kQ)$:
\begin{eqnarray} \label{eqn.k0rel}
[\sP_i(z)]=\sum_{j\in Q_0}a_{ij}[\sP_j(z-1)].
\end{eqnarray} 
Let $\sP(z)=([\sP_1(z)],\ldots,[\sP_n(z)])^T$, the relation in (\ref{eqn.k0rel}) can be written more succinctly as
$$
\sP(z)=M\sP(z-1).
$$

Consider all the non cyclic vertices $Q_0\-\{\nu_1,\ldots,\nu_p\}$. Let $i$ be a non cyclic vertex and consider the set $\{p_1^i,\ldots,p_{l_i}^i\}$ of all paths that start at $i$ and do not pass through any cyclic vertex until ending at a cyclic vertex. More explicitly, $\{p_1^i,\ldots,p_{l_i}^i\}$ are all paths which satisfy:
\begin{enumerate}
\item[(i)] $s(p_j^i)=i$.
\item[(ii)] $t(p_j^i)$ is a cyclic vertex.
\item[(iii)] If $|p_j^i|>1$ and $a$ is an arrow in $p_j^i$ such that $p_j^i\neq qa$ for any path $q$, then $t(a)$ is not a cyclic vertex.
\end{enumerate}
For example, consider the quiver
$$
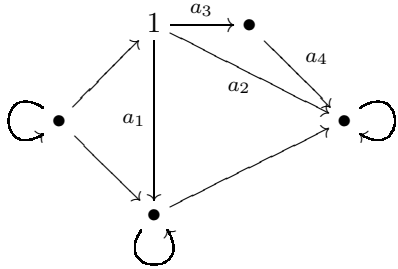
\UseComputerModernTips
\xymatrix{ {} & {1}\ar[r]^{a_3} \ar[drr]_{a_2} \ar[dd]_{a_1}& {\bullet} \ar[dr]^{a_4} & {} \\
           {\bullet}\ar@(ul,dl)[] \ar[ur] \ar[dr] & {} & {} & {\bullet} \ar@(ur,dr)[] \\
           {} & {\bullet} \ar@(dl,dr)[] \ar[urr] & {} & {} \\
           {} & {} & {} & {} & {} }
$$
The vertex $1$ is not cyclic and all paths which satisfy the above 3 conditions are $\{a_1,a_2,a_3a_4\}$.

For the non cyclic vertex $i$ consider the submodule of $e_ikQ$ given by 
$$
\bigoplus_{j=1}^{l_i}p_j^ikQ.
$$
Let $m=\max\{|p_1^i|,\ldots,|p_{l_i}^i|\}$ where $|p_j^i|$ is the length of the path $p_j^i$. Then every path in $Q$ starting at $i$ of length at least $m$ has the form $p_j^iq$ for some path $q$. Hence, 
$$
(e_ikQ)_{\geq m} \subseteq \bigoplus_{j=1}^{l_i}p_j^ikQ \subseteq e_ikQ. 
$$
As $e_ikQ/(e_ikQ)_{\geq m}$ has finite dimension, $e_ikQ/\bigoplus_{j=1}^{l_i}p_j^ikQ$ has finite dimension. Therefore, $\pi^*e_ikQ\cong \bigoplus_{j=1}^{l_i}\pi^*p_j^ikQ$ in $\QGr kQ$.

Since $p_j^ikQ\cong e_{t(p_j^i)}kQ(-|p_j^i|)$, we get
$$
\pi^*e_ikQ\cong \bigoplus_{j=1}^{l_i}\pi^*e_{t(p_j^i)}kQ(-|p_j^i|).
$$
in $\QGr kQ$. By shifting we get, for any $z\in \ZZ$, an isomorphism in $\QGr kQ$ 
$$
\pi^*e_ikQ(z)\cong \bigoplus_{j=1}^{l_i}\pi^*e_{t(p_j^i)}kQ(z-|p_j^i|).
$$
Again, by Proposition 3.2 of \cite{Sm1}, given any finitely generated projective $\sM \in \QGr kQ$, there is an isomorphism
$$
\sM \cong \bigoplus_{i\in Q_0}\left( \bigoplus_{j\in J_i} \pi^*e_ikQ(-n_{ij})^{m_{ij}}\right).
$$
However, for every non cyclic vertex $i$, the summands $\pi^*e_ikQ(-n_{ij})$ will be isomorphic to a sum of shifts of the projectives $\{\pi^*e_{\nu_l}kQ\}$ for cyclic vertices $\nu_l$. Hence, given any finitely generated projective $\sM \in \QGr kQ$, there is an isomorphism
$$
\sM \cong \bigoplus_{i=1}^{p}\left( \bigoplus_{j\in J_i} \pi^*e_{\nu_i}kQ(-n_{ij})^{m_{ij}}\right).
$$ 
Therefore, the group $\K0(\QGr kQ)$ is generated by all the elements 
$$
\{[\pi^*e_{\nu_i}kQ(z)]=[\sP_{\nu_i}(z)]\;|\; 0\leq i\leq p, z\in \ZZ\}.
$$   

The relation (\ref{eqn.k0rel}) requires using the elements $[\sP_i(z)]$ for all vertices $i$, not just the cyclic vertices. However, as $NR=RM$ we get
$$
R\sP(z)=RM\sP(z-1)=NR\sP(z-1).
$$
Now $R_{\nu_i}\sP(z)$ is the $\nu_i$-th entry of $M\sP(z)$ which is just $[\sP_{\nu_i}(z+1)]$, and hence, 
$$
R\sP(z)=([\sP_{\nu_1}(z+1)],\ldots,[\sP_{\nu_p}(z+1)])^T.
$$
If we let $\sP'(z)=([\sP_{\nu_1}(z)],\ldots,[\sP_{\nu_p}(z)])^T$ then we get a way to relate the $[\sP_{\nu_i}(z)]$ with the $[\sP_{\nu_i}(z-1)]$ with out using the other elements $[\sP_i(z-1)]$ where $i$ is a non cyclic vertex, namely:
\begin{eqnarray*} 
\sP'(z+1)=N\sP'(z).
\end{eqnarray*}
As $N$ is an invertible matrix over $\ZZ$, we also get $\sP'(z)=N^{-1}\sP'(z+1)$. It follows by induction that
\begin{eqnarray} \label{eqn.k0rel2}
\sP'(z)=N^{z}\sP'(0)
\end{eqnarray}
for all $z\in \ZZ$. As $\sP'(0)=([\sP_{\nu_1}],\ldots,[\sP_{\nu_p}])^T$ we find that
$$
\{[\sP_{\nu_1}],\ldots,[\sP_{\nu_p}]\}
$$
is a set of generators for $\K0(\QGr kQ)$. Since $\K0(\QGr kQ)$ is a free abelian group of rank $p$ and the elements $\{[\sP_{\nu_1}],\ldots,[\sP_{\nu_p}]\}$ generate $\K0(\QGr kQ)$ they are a $\ZZ$-basis for $\K0(\QGr kQ)$. Thus, we have established the following proposition:

\begin{Prop}
$\K0(\QGr kQ)$ is a free abelian group with basis 
$$
\{[\sP_{\nu_1}],\ldots,[\sP_{\nu_p}]\}. 
$$
Moreover, every finitely generated projective $\sM \in \QGr kQ$ is isomorphic to a sum of shifts of the objects $\{\sP_{\nu_1},\ldots,\sP_{\nu_p}\}$. 
\end{Prop}

Let $\psi:\K0(\QGr kQ) \to \ZZ^p$ be the isomorphism $\psi([\sP_{\nu_i}])=e_i$, where $e_i$ is the usual $i$-th basis vector in $\ZZ^p$. As was already noted, every element in $\QGr kQ$ is isomorphic to a direct sum of $\sP_{\nu_i}(z)$ and these are related by equation (\ref{eqn.k0rel2}). Hence, to understand the positive cone $\K0^+(\QGr kQ)$, we need to understand the matrices $N^z$ for $z\in \ZZ$.   

The matrix $N$ was defined so that $NR=RM$ where $R$ is the $p\times n$ matrix whose $i$-th row is $R_{\nu_i}$ the $\nu_i$-th row of the matrix $M$. The elements $b_{ij}$ of $N$ satisfy 
\begin{equation} \label{eqn.bij}
R_{\nu_i}(2)=\sum_{j=1}^{p}b_{ij}R_{\nu_j}.	
\end{equation}
Moreover, $b_{ij}=0$ if $j<i$, and $b_{ii}=1$. Using this we can recursively solve for the $b_{ij}$ in terms of the $a_{ij}$. Comparing the $k$-th entries of both sides of \ref{eqn.bij} yields
$$
a_{\nu_i,k}(2)=\sum_{j=1}^{p} b_{ij}a_{\nu_jk}=a_{\nu_ik}+\sum_{j=i+1}^{p}b_{ij}a_{\nu_jk}.
$$ 
If we pick $k=\nu_{i+1}$, then $a_{\nu_j,\nu_{i+1}}=0$ for $j>i+1$ and hence
$$
a_{\nu_i,\nu_{i+1}}(2)=a_{\nu_i,\nu_{i+1}}+b_{i,i+1}a_{\nu_{i+1},\nu_{i+1}}
$$
giving
$$
b_{i,i+1}=a_{\nu_i,\nu_{i+1}}(2)-a_{\nu_i,\nu_{i+1}}.
$$ 
More generally, picking $k=\nu_{i+l}$ gives
$$
a_{\nu_i,\nu_{i+l}}(2)=\sum_{j=1}^{p}b_{ij}a_{\nu_j,\nu_{i+l}}.
$$
Since $a_{\nu_j,\nu_{i+l}}=0$ if $j>i+l$, $b_{i,j}=0$ if $j<i$ and $b_{ii}=1$, the sum reduces to
$$
a_{\nu_i,\nu_{i+l}}(2)=a_{\nu_i,\nu_{i+l}}+\sum_{j=i+1}^{i+l}b_{i,j}a_{\nu_{j},\nu_{i+l}}=a_{\nu_i,\nu_{i+l}}+\sum_{j=1}^{l}b_{i,i+j}a_{\nu_{i+j},\nu_{i+l}}.
$$
Hence, because $a_{\nu_{i+l},\nu_{i+l}}=1$, we can solve for $b_{i,i+l}$ to get
$$
b_{i,i+l}=a_{\nu_i,\nu_{i+l}}(2)-a_{\nu_i,\nu_{i+l}}-\sum_{j=1}^{l-1}b_{i,i+j}a_{\nu_{i+j},\nu_{i+l}}.
$$

\begin{Prop}
For the moment, let $a_{i,j}$ denote $a_{\nu_i,\nu_j}$. The elements $b_{i,i+l}$ of the matrix $N$ are given by $b_{i,i+1}=a_{i,i+1}(2)-a_{i,i+1}$ and for $l>1$,
\begin{eqnarray} \label{eqn.awful}
b_{i,i+l}=(a_{i,i+l}(2)-a_{i,i+l})+ \\
\sum_{d=1}^{l-1}\sum_{1\leq j_1<j_2<\cdots<j_d\leq l-1}{(-1)^d(a_{i,i+j_1}(2)-a_{i,i+j_1})a_{i+j_1,i+j_2}a_{i+j_2,i+j_3}\cdots a_{i+j_d,i+l}}. \nonumber
\end{eqnarray}
\end{Prop}
\begin{proof}
The case $l=1$ was already observed. Suppose the formula holds for all $1\leq j<l$ and $i+l\leq p$. $b_{i,i+l}$ is related to $b_{i,i+j}$ for $1\leq j<l$ by the relation
\begin{eqnarray*}
b_{i,i+l}=a_{i,i+l}(2)-a_{i,i+l}-\sum_{j=1}^{l-1}b_{i,i+j}a_{i+j,i+l}. 
\end{eqnarray*}
By the inductive hypothesis, each $b_{i,i+j}$ term has the form \ref{eqn.awful}. Therefore,
\begin{eqnarray*}
b_{i,i+j}a_{i+j,i+l}=(a_{i,i+j}(2)-a_{i,i+j})a_{i+j,i+l}+ \\
\sum_{d=1}^{j-1}\sum_{1\leq j_1<\cdots<j_d\leq j-1}{(-1)^{d}(a_{i,i+j_1}(2)-a_{i,i+j_1})a_{i+j_1,i+j_2}\cdots a_{i+j_d,i+j}}a_{i+j,i+l}. 
\end{eqnarray*}
Hence,
\begin{eqnarray*}
\sum_{j=1}^{l-1}{b_{i,i+j}a_{i+j,i+l}}=\sum_{j=1}^{l-1}{(a_{i,i+j}(2)-a_{i,i+j})a_{i+j,i+l}}+\\
\sum_{j=2}^{l-1}\sum_{d=1}^{j-1}\sum_{1\leq j_1<\cdots<j_d\leq j-1}{(-1)^d(a_{i,i+j_1}(2)-a_{i,i+j_1})a_{i+j_1,i+j_2}\cdots a_{i+j_d,i+j}a_{i+j,i+l}}.
\end{eqnarray*}
Using the fact that $\sum_{j=2}^{l-1}\sum_{d=1}^{j-1}=\sum_{d=1}^{l-2}\sum_{j=d+1}^{l-1}$,  changing some summation variables and inputting a few extra minus signs we can rewrite the above equation as 
\begin{eqnarray*}
\sum_{j=1}^{l-1}{b_{i,i+j}a_{i+j,i+l}}=-\sum_{j_1=1}^{l-1}{(-1)^1(a_{i,i+j_1}(2)-a_{i,i+j_1})a_{i+j_1,i+l}}-\\
\sum_{d=1}^{l-2}\sum_{j=d+1}^{l-1}\sum_{1\leq j_1<\cdots<j_d\leq j-1}{(-1)^{d+1}(a_{i,i+j_1}(2)-a_{i,i+j_1})a_{i+j_1,i+j_2}\cdots a_{i+j_d,i+j}a_{i+j,i+l}}
\end{eqnarray*}
Thinking of $j=j_{d+1}$ we can write the sum $\sum_{j=d+1}^{l-1}\sum_{1\leq j_1<\cdots<j_d\leq j-1}$ as $\sum_{1\leq j_1<\cdots<j_d<j_{d+1}\leq l-1}$. Hence,
\begin{eqnarray*}
\sum_{j=1}^{l-1}{b_{i,i+j}a_{i+j,i+l}}=-\sum_{j_1=1}^{l-1}{(-1)^1(a_{i,i+j_1}(2)-a_{i,i+j_1})a_{i+j_1,i+l}}-\\
\sum_{d=1}^{l-2}\sum_{1\leq j_1<\cdots<j_{d+1}\leq l-1}{(-1)^{d+1}(a_{i,i+j_1}(2)-a_{i,i+j_1})a_{i+j_1,i+j_2}\cdots a_{i+j_d,i+j_{d+1}}a_{i+j_{d+1},i+l}} 
\end{eqnarray*}
We can rewrite the equation above as
\begin{eqnarray*}
\sum_{j=1}^{l-1}{b_{i,i+j}a_{i+j,i+l}}=-\sum_{j_1=1}^{l-1}{(-1)^1(a_{i,i+j_1}(2)-a_{i,i+j_1})a_{i+j_1,i+l}}-\\
\sum_{d=2}^{l-1}\sum_{1\leq j_1<\cdots<j_d\leq l-1}{(-1)^{d}(a_{i,i+j_1}(2)-a_{i,i+j_1})a_{i+j_1,i+j_2}\cdots a_{i+j_{d-1},i+j_d}a_{i+j_d,i+l}} \\
=-\sum_{d=1}^{l-1}\sum_{1\leq j_1<\cdots<j_d\leq l-1}{(-1)^{d}(a_{i,i+j_1}(2)-a_{i,i+j_1})a_{i+j_1,i+j_2}\cdots a_{i+j_{d-1},i+j_d}a_{i+j_d,i+l}}
\end{eqnarray*}
Therefore,
\begin{eqnarray*}
b_{i,i+l}=a_{i,i+l}(2)-a_{i,i+l}-\sum_{j=1}^{l-1}{b_{i,i+j}a_{i+j,i+l}}=a_{i,i+l}(2)-a_{i,i+l}+ \\
\sum_{d=1}^{l-1}\sum_{1\leq j_1<\cdots<j_d\leq l-1}{(-1)^{d}(a_{i,i+j_1}(2)-a_{i,i+j_1})a_{i+j_1,i+j_2}\cdots a_{i+j_d,i+l}}
\end{eqnarray*}
showing the formula holds for $b_{i,i+l}$. 
\end{proof}

\begin{Prop} \label{prop.minnu}
If $\nu_i$ and $\nu_{i+l}$ are incomparable, then $b_{i,i+l}=0$. If $\nu_{i+l}$ is minimal over $\nu_i$ then
$$
b_{i,i+l}=a_{\nu_i,\nu_{i+l}}(2)-a_{\nu_i,\nu_{i+l}}>0.
$$
\end{Prop}
\begin{proof}
Recall that we are working with a quiver $Q$ in which every simple cycle is a loop, and for which there is a path between two distinct cyclic vertices if and only if there is an arrow between them. Hence, $\nu_i \preceq \nu_j$ if and only if $a_{i,i+l}\neq 0$. Recall we are writing $a_{i,i+l}$ for $a_{\nu_i,\nu_{i+l}}$. 

If $\nu_i$ and $\nu_{i+l}$ are incomparable, then $a_{i,i+l}(m)=0$ for all $m$ as there are no paths from $\nu_i$ to $\nu_{i+l}$.

Consider a typical term in the sum in equation \ref{eqn.awful}:
$$
a_{i,i+j_1}(n)a_{i+j_1,i+j_2}\cdots a_{i+j_d,i+l}
$$
where $n$ is either $1$ or $2$. If this is not zero then there is a path of the form
$$
\UseComputerModernTips
\xymatrix{ {\nu_i}\ar[r] & {\nu_{i+j_1}}\ar[r] & {\cdots}\ar[r] & {\nu_{i+j_d}} \ar[r]& {\nu_{i+l}} }
$$
which shows $\nu_i\prec \nu_{i+j_1}\prec \cdots \prec \nu_{i+l}$ which is a contradiction. Hence, every term in the sum in equation \ref{eqn.awful} is zero. As $a_{i,i+l}(m)=0$ for all $m$ we deduce from equation \ref{eqn.awful} that $b_{i,i+l}=0$. 

Suppose $\nu_{i+l}$ is minimal over $\nu_i$. Since $\nu_i \preceq \nu_{i+l}$, there is an arrow from $\nu_i$ to $\nu_{i+l}$. For every such arrow $a$, we get two paths from $\nu_i$ to $\nu_{i+l}$ of length $2$ by either first traversing the loop at $\nu_i$ then followed by $a$ or traversing $a$ followed by the loop at $\nu_{i+l}$. Hence, $a_{\nu_i,\nu_{i+l}}(2)\geq 2a_{\nu_i,\nu_{i+l}}$ which implies $a_{\nu_i,\nu_{i+l}}(2)-a_{\nu_i,\nu_{i+l}}\geq a_{\nu_i,\nu_{i+1}}>0$.

Again, if any term $a_{i,i+j_1}(n)a_{i+j_1,i+j_2}\cdots a_{i+j_d,i+l}$ in the sum of equation \ref{eqn.awful} is not zero, then there must be a path from $\nu_i$ to $\nu_{i+l}$ which passes through another cyclic vertex. This is a contradiction to the assumption that $\nu_{i+l}$ is minimal over $\nu_i$. Hence, every term in the sum is zero and we are left with
$$
b_{i,i+l}=a_{\nu_i,\nu_{i+l}}(2)-a_{\nu_i,\nu_{i+l}}>0.
$$
\end{proof}

The binomial coefficients ${m\choose i}$, for a fixed $i$, determine polynomial functions with rational coefficients:
$$
{m\choose i}=\frac{m(m-1)\cdots (m-i+1)}{i!}.
$$
Let $f_i(m)$ denote this polynomial function. Then $f_i$ has degree $i$ and the roots are precisely $\{0,1,\ldots,i-1\}$.

Consider the $n\times n$ Jordan block matrix 
$$
J(1,n)=\bmat
1 & 1 & 0 &\cdots & 0 \\
0 & 1 & 1 & \cdots & 0 \\
  & & & \ddots & \\
0 & 0 & 0 & \cdots & 1	
\emat.
$$
For any function $g$ analytic around $1$, the $ij$-th entry of $g(J(1,n))$ is $0$ if $i>j$ and $g^{(j-i)}(1)/(j-i)!$ if $i\leq j$. In particular, if $g(t)=t^z$ for $z\in \ZZ$, then 
$$
\frac{1}{i!}\left(\frac{d}{dt}\right)^i t^z|_{t=1}=\frac{z(z-1)\cdots (z-i+1)}{i!}=f_i(z).
$$
Therefore, for all $z\in \ZZ$, $(J(1,n)^z)_{ij}=0$ if $i>j$ while
$$
(J(1,n)^z)_{ij}=f_{j-i}(z)
$$
for $i\leq j$.

For the matrix $N=(b_{ij})$, all the eigenvalues of $N$ are $1$ and $N$ is upper triangular. Hence, we can put $N$ into Jordan canonical form over $\QQ$. Therefore, there is an invertible matrix $S\in M_p(\QQ)$ such that $N=S^{-1}\left(J(1,n_1)\oplus \cdots \oplus J(1,n_l)\right)S$ where $n_1+\cdots+n_l=p$. Therefore,
$$
N^z=S^{-1}\left(J(1,n_1)^z\oplus \cdots \oplus J(1,n_l)^z\right)S.
$$
Since the entries of $J(1,n_i)^z$ are rational polynomial functions of $z$ and the entries of $S$ and $S^{-1}$ are rational, we determine that the entries, $b_{ij}(z)$, of $N^z$ are rational polynomials of $z$.

The matrix $N$ can be written as $N=I+U$ where $I$ is the identity matrix and $U$ is a strictly upper triangular matrix. Since $U$ and $I$ commute, $U^m=0$ for $m\geq p$ and ${m\choose l}=0$ if $m<l$, we can write
$$
N^m=(I+U)^m=\sum_{l=0}^{m}{m\choose l}U^l=\sum_{l=0}^{p-1}{m\choose l}U^l.
$$
If we let $u_{ij}(l)$ be the $ij$-th entry of $U^l$ then we determine
$$
b_{ij}(m)=\sum_{l=1}^{p-1}{m\choose l}u_{ij}(l)=\sum_{l=0}^{p-1}f_l(m)u_{ij}(l)	
$$ 
which is a rational polynomial function in $m$. Since $b_{ij}(z)$ and $\sum_{l=0}^{p-1}f_l(z)u_{ij}(l)$ are rational polynomials and $b_{ij}(m)=\sum_{l=0}^{p-1}f_l(m)u_{ij}(l)$ for all $m\in \NN$, 
$$
b_{ij}(z)=\sum_{l=0}^{p-1}f_{l}(z)u_{ij}(l).
$$
As $N$ is upper triangular we know $b_{ij}(z)=0$ if $i>j$. We can also see that $b_{ij}(z)$ is a polynomial of degree at most $j-i$ since $u_{ij}(l)=0$ if $l>j-i$.

\begin{Prop} \label{prop.bij(z)}
If $\nu_i$ and $\nu_{i+l}$ are incomparable, then $b_{i,i+l}(z)=0$. If $\nu_{i+l}$ is minimal over $\nu_i$, then $b_{i,i+l}(z)=b_{i,i+l}\cdot z=(a_{\nu_i,\nu_{i+l}}(2)-a_{\nu_i,\nu_{i+l}})z$.
\end{Prop}
\begin{proof}
If $\nu_{i+l}$ is minimal over $\nu_i$ then we know $b_{i,i+l}=a_{\nu_i,\nu_{i+l}}(2)-a_{\nu_i,\nu_{i+l}}>0$ by Proposition \ref{prop.minnu}.

The elements $u_{ij}(m)$, for $m\geq 2$, are given by
\begin{eqnarray*}
u_{ij}(m)=\sum_{k_1=1}^{p}\cdots \sum_{k_{m-1}=1}^{p}u_{i,k_1}u_{k_1,k_2}\cdots u_{k_{m-1},j}.
\end{eqnarray*}
The entries of $U$ satisfy $u_{ij}=b_{ij}$ if $i<j$ while $u_{ij}=0$ if $i\geq j$. Hence, in a typical term $u_{i,k_1}u_{k_1,k_2}\cdots u_{k_{m-1},j}$, if $k_1\leq i$, then $u_{i,k_1}=0$, if $k_{m-1}\geq j$ then $u_{k_{m-1},j}=0$ or if $k_{i+1}\leq k_{i}$ then $u_{k_i,k_{i+1}}=0$. Hence, the only time the term is not necessarily zero is when $i<k_1<\cdots<k_{m-1}<j$ in which case
$$
u_{i,k_1}u_{k_1,k_2}\cdots u_{k_{m-1},j}=b_{i,k_1}b_{k_1,k_2}\cdots b_{k_{m-1},j}.
$$  

If $\nu_i$ and $\nu_{i+l}$ are incomparable then for any $i<k_1<\cdots <k_{m-1}<i+l$, the term
$$
b_{i,k_1}b_{k_1,k_2}\cdots b_{k_{m-1},j}
$$
must be zero. If not, then by Proposition \ref{prop.minnu}, it follows that $\nu_i\prec \nu_{k_1}\prec \cdots \prec \nu_{k_{m-1}} \prec \nu_{i+l}$ which is a contradiction as $\nu_i \not\prec \nu_{i+l}$. Hence, $u_{i,i+l}(m)=0$ for $m\geq 2$. However, $u_{i,i+l}(1)=u_{i,i+l}=b_{i,i+l}=0$ and $u_{i,i+l}(0)=0$ also holds. Hence,
$$
b_{i,i+l}(z)=\sum_{j=0}^{p-1}f_{j}(z)u_{i,i+l}(j)=0.
$$

In the case where $\nu_{i+l}$ is minimal over $\nu_i$, we have 
$$
b_{i,k_1}b_{k_1,k_2}\cdots b_{k_{m-1},i+l}=0
$$
for otherwise, $b_{i,k_1},b_{k_i,k_{i+1}},b_{k_{m-1},i+l}\neq 0$ implies $\nu_i\prec \nu_{k_1}\prec \cdots \prec \nu_{i+l}$ which contradicts the minimality of $\nu_{i+l}$ over $\nu_i$. Thus, $u_{i,i+l}(m)=0$ for $m\geq 2$. Therefore,
$$
b_{i,i+l}(z)=\sum_{m=0}^{p-1}f_m(z)u_{i,i+l}(m)=u_{i,i+l}(0)+f_1(z)u_{i,i+l}(1)=zb_{i,i+l}.
$$
\end{proof}

Recall, the group $\K0(\QGr kQ)$ is free with basis 
$$
\{[\sP_{\nu_1}],\ldots,[\sP_{\nu_p}]\}
$$
and there are the relations
$$
\sP'(z)=N^z\sP'(0)
$$ 
where $\sP'(z)=([\sP_{\nu_1}(z)],\ldots,[\sP_{\nu_p}(z)])^T$. Hence, in $\K0(\QGr kQ)$,
$$
[\sP_{\nu_i}(z)]=\sum_{j=1}^{p}b_{ij}(z)[\sP_{\nu_j}].
$$

It was already observed that every finitely generated projective $\sM \in \QGr kQ$ is isomorphic to a direct sum of objects of the form $\sP_{\nu_i}(z)=\pi^*e_{\nu_i}kQ(z)$. Hence, the positive cone of $\K0(\QGr kQ)$ is generated as a monoid by the elements $\{[\sP_{\nu_i}(z)]\;|\; z\in \ZZ, i=1,\ldots,p\}$.

We have an isomorphism $\psi:\K0(\QGr kQ)\to \ZZ^p$ which sends $[\sP_{\nu_i}]$ to $e_i$. It will now be shown that under this isomorphism, $\psi(\K0(\QGr kQ)^+)=\Delta(E_Q)$.

Recall that $\Delta(E_Q)$ is the submonoid of $\ZZ^p$ generated by elements of the form 
$$
e_i+\sum_{i\prec k}z_ke_k
$$
where $z_k\in \ZZ$. Since $b_{ij}(z)=0$ if $i>j$, $b_{ii}(z)=1$ and $b_{ij}(z)=0$ if $\nu_i$ and $\nu_j$ are incomparable, we determine
$$
\psi([\sP_{\nu_i}(z)])=\sum_{j=1}^{p}b_{ij}(z)\psi([\sP_{\nu_j}])=e_i+\sum_{i\prec j}b_{ij}(z)e_{j}\in \Delta(E_Q).
$$
Hence, $\psi(\K0^+(\QGr kQ))\subseteq \Delta(E_Q)$.

Since $E_Q$ is a finite poset, it has maximal elements. Let $L_1$ be the set of all maximal elements of $E_Q$. Inductively define $L_n$ to be the set of maximal elements of the poset $E_Q\-(L_1\cup \cdots \cup L_{n-1})$ if the latter is non empty, otherwise $L_n=\emptyset$.  

Let $\bv=(v_1,\ldots,v_n)\in \Delta(E_Q)$. If $\supp(\bv) \subset L_1$, then every element in $\supp(\bv)$ is minimal in $\supp(\bv)$ and we can write $\bv=\sum _{i\in \supp(\bv)}n_ie_i$ where $n_i>0$. Hence, the element $[\sM]=[\bigoplus_{i\in \supp(\bv)} \sP_{\nu_i}^{n_i}]$ is in $K_0^+(\QGr kQ)$ and $\psi([\sM])=\bv$.

Suppose we have shown every element $\bv \in \Delta(E_Q)$ for which $\supp(\bv) \subset L_1\cup \cdots \cup L_r$ is in $\psi(K_0^+(\QGr kQ))$. 

Consider an element 
$$
n_ie_i+\sum_{i\prec k}z_ke_k \in \Delta(E_Q)
$$
where $i\in L_{r+1}$ and $n_i\geq 1$. Let $i_1,\ldots,i_l$ be all the elements in $E_Q$ which are minimal over $i$. Since $\nu_{i_j}$ is minimal over $\nu_i$, Propositions \ref{prop.minnu} and \ref{prop.bij(z)} imply $b_{i,i_j}(z)=b_{i,i_j}z$ and $b_{i,i_j}=a_{\nu_i,\nu_{i_j}}(2)-a_{\nu_i,\nu_{i_j}}>0$. Hence, we can find a $z\in \ZZ$ such that $z_{i_j}-b_{i,i_j}(z)>0$ for all $i_1,\ldots,i_l$. 

Write
$$
n_ie_i+\sum_{i\prec k}z_ke_k=(n_i-1)e_i+e_i+\sum_{i\prec k}b_{ik}(z)e_k+\sum_{i\prec k}(z_k-b_{ik}(z))e_k.
$$
Notice the support of the vector $\bv=(v_1,\ldots,v_p)=\sum_{i\prec k}(z_k-b_{ik}(z))e_k$ is contained in $(i,\infty)=\{k\;|\; i\prec k\} \subset L_1\cup \cdots \cup L_r$. Also, by the choice of $z$, the term $v_{i_j}=z_{i_j}-b_{i,i_j}(z)>0$ for all $i_j$ minimal over $i$. By Lemma \ref{lem.altdelta}, $\bv\in \Delta(E_Q)$ since $v_k>0$ for all $k$ minimal in $\supp(\bv)$. 

Since $\bv \in \Delta(E_Q)$ and $\supp(\bv)\subset L_1\cup \cdots \cup L_r$, we know $\bv=\psi([\sM])$ for some finitely generated projective object $\sM \in \QGr kQ$. Therefore,
$$
\psi([\sP_{\nu_i}^{n_i-1}\oplus \sP_{\nu_i}(z)\oplus \sM])=(n_i-1)e_i+e_i+\sum_{i\prec k}b_{ik}(z)e_k+\bv=n_ie_i+\sum_{i\prec k}z_ke_k.
$$

Thus, every element of the form $n_ie_i+\sum_{i\prec k}z_ke_k$ with $i\in L_{r+1}$ is in $\psi(\K0^+(\QGr kQ))$. If $\bv \in \Delta(E_Q)$ and $\supp(\bv) \subset L_1\cup \cdots \cup L_r\cup L_{r+1}$, then we can write $\bv=\sum_{i\in L_1\cup \cdots \cup L_{r+1}}(n_ie_i+\sum_{i\prec k}z^i_ke_k)$ where the $n_i \geq 0$ and $z^i_k=0$ for all $i\prec k$ if $n_i=0$. Since every element of the form $n_ie_i+\sum_{i\prec k}z_k^ie_k$ appearing in $\bv$ is in $\psi(\K0^+(\QGr kQ))$, we get $\bv\in \psi(\K0^+(\QGr kQ))$. Therefore, every vector $\bv\in \Delta(E_Q)$ with $\supp(\bv)\subset L_1\cup \cdots \cup L_{r+1}$, is in the image of $\K0^+(\QGr kQ)$.

Since $E_Q \subset L_1\cup \cdots \cup L_n$ for $n\gg 0$, we get by induction that $\Delta(E_Q)\subset \psi(\K0^+(\QGr kQ))$.

Thus, $\psi:\K0(\QGr kQ) \to \ZZ^p$ is a group isomorphism such that 
$$
\psi(\K0^+(\QGr kQ))=\Delta(E_Q)
$$ 
and Theorem \ref{thm.k0qgrkQ} is proved.

\section{An explicit equivalence.}
\subsection{}
The proof of Theorem \ref{thm.main3Q} using Grothendieck groups gives the existence of an equivalence $\QGr kQ \equiv \QGr kQ'$ when $E_Q=E_{Q'}$. This Section gives an explicit description of an equivalence.  

Given a path algebra $kQ$ of finite GK-dimension, the element $(1,1,\cdots,1) \in \Delta(E_Q)\subset Z^p$ is an order unit. To see this just note that for any $(v_1,\ldots,v_p)\in \ZZ^p$, 
$$
n(1,\ldots,1)-(v_1,\ldots,v_p)\in \NN^p \subset \Delta(E_Q)
$$ 
for $n\gg 0$. 

The isomorphism from $\K0(\QGr kQ) \to \ZZ^p$ was constructed by first using the Veronese equivalence $V:\QGr kQ \to \QGr kQ^{(n)}$ for a sufficient choice of $n$ and then mapping the elements $[\pi^*e_{\nu_i}kQ^{(n)}] \to e_i$ where the $\nu_i$ are the cyclic vertices of $Q$. Let $P=e_{\nu_1}kQ\oplus \cdots \oplus e_{\nu_p}kQ\in \Gr kQ$. Then under the isomorphism $\K0(\QGr kQ) \to \ZZ^p$;
$$
\UseComputerModernTips
\xymatrix{{[\oplus_{i=1}^p \pi^*e_{\nu_i}kQ]} \ar@{|->}[r]^{V} & {[\oplus_{i=1}^p \pi^*e_{\nu_i}kQ^{(n)}]} \ar@{|->}[r]^(0.6){\psi} & {(1,\ldots,1)} }
$$
Therefore, the element $[\pi^*P]$ is an order unit in $\K0(\QGr kQ)$. 

For an ultramatricial algebra $A$, a finitely generated projective module is an order unit in $\K0(A)$ if and only if it is a generator, see Chapter $20$ of \cite{Good}. Hence, $\pi^*P$ is a finitely generated projective generator in the category $\QGr kQ\equiv \Mod S(Q)$ which implies there is an equivalence of categories
$$
\Hom_{\QGr kQ}(\pi^*P,-):\QGr kQ \to \Mod \End_{\QGr kQ}(\pi^*P).
$$

Following the ideas in \cite{Sm1}, we can show $\End_{\QGr kQ}(\pi^*P)$ is an ultramatricial algebra.

A basis of the graded vector space $P$ consists of all paths in $Q$ that begin at a cyclic vertex. The torsion elements of $P$ are all paths which end at a vertex $v$ which is the source of only finitely many paths. Hence, the torsion submodule $\tau P$ has a basis consisting of all paths in $Q$ that start at a cyclic vertex and end at a vertex $v$ which is the source of only finitely many paths. The natural quotient map $P \to P/\tau P$ is an isomorphism in $\QGr kQ$ and thus induces an isomorphism of algebras
$$
\End_{\QGr kQ}(\pi^*P) \cong \End_{\QGr kQ}(\pi^* (P/\tau P)).
$$
Write $P'$ for $P/\tau P$. 

By definition,
$$
\End_{\QGr kQ}(\pi^*(P'))=\varinjlim \Hom_{\Gr kQ}(M,P'/N)
$$
where the direct limit is over all pairs $M,N$ such that $P'/M$ and $N$ are torsion. Since $P'$ is a finitely generated module, if $P'/M$ is torsion then $P'/M$ must be finite-dimensional. Thus, $M\supset P'_{\geq n}$ for some $n$. As $P'$ is torsion free it follows that
$$
\End_{\QGr kQ}(\pi^*P')=\varinjlim_{n}\Hom_{\Gr kQ}(P'_{\geq n},P')
$$
where the structure maps in the direct limit are given by restriction of morphisms.

Since graded morphisms must preserve degree, 
$$
\Hom_{\Gr kQ}(P'_{\geq n},P') \cong \Hom_{\Gr kQ}(P'_{\geq n},P'_{\geq n}))
$$ 
as vector spaces and thus there is an isomorphism of vector spaces
$$
\End_{\QGr kQ}(\pi^*P') \cong \varinjlim_{n}\Hom_{\Gr kQ}(P'_{\geq n},P'_{\geq n})
$$ 
However, from the definition of composition in $\QGr kQ$ it can be checked that this isomorphism of vector spaces respects composition and is thus a $k$-algebra isomorphism. 

As $P'_{\geq n}=P'_{n}kQ$, each morphism $\phi \in \Hom_{\Gr kQ}(P'_{\geq n},P'_{\geq n})$ is completely determined by $\phi|_{P'_{n}}:P'_n \to P'_n$ which is a morphism of right $kQ_0$-modules. Hence, there is an injective $k$-algebra map 
$$
\Hom_{\Gr kQ}(P'_{\geq n},P'_{\geq n}) \to \Hom_{kQ_0}(P'_{n},P'_{n}).
$$

Suppose $\phi \in \Hom_{kQ_0}(P'_n,P'_n)$. Given a path $p=a_1\cdots a_n\cdots a_{n+i}$ of length $n+i$ which is not in $\tau P$, define
$$
\phi(p)=\phi(a_1\cdots a_n)a_{n+1}\cdots a_i \in P'.
$$ 
Since paths in $P$ not in $\tau P$ form a basis for $P'=P/\tau P$, this extension of $\phi$ defines a graded vector space map $\phi:P'_{\geq n} \to P'_{\geq n}$ which can be seen to be morphism of graded $kQ$-modules. Hence, the natural map $\Hom_{\Gr kQ}(P'_{\geq n}) \to \Hom_{kQ_0}(P'_n,P'_n)$ is an isomorphism.

As a vector space, we can write $P_n'=\oplus_{i\in Q_0} P_n'e_i$. Any right $kQ_0$ module morphism $\phi:P_n'\to P_n'$ must satisfy $\phi(P_n'e_i)\subset P_n'e_i$. Hence, any $\phi \in \End_{kQ_0}(P'_n)$ is completely determined by linear transformations $\phi_i:P_n'e_i \to P_n'e_i$. Moreover, given linear maps $\psi_i:P_n'e_i \to P_n'e_i$, there is a right $kQ_0$-module homomorphism $\psi:P_n'\to P_n'$ such that $\psi|_{P_n'e_i}=\psi_i$. Hence, $\Hom_{kQ_0}(P'_n,P'_n)$ is a product of matrix algebras.  

Thus, the endomorphism algebra $\End_{\QGr kQ}(\pi^*P) \cong \End_{\QGr kQ}(\pi^* P')$ is a direct limit of matricial algebras and is thus an ultramatricial algebra.

\begin{Prop}
Let $kQ$ be a path algebra of finite GK-dimension. Let $P=\oplus e_{\nu_i}kQ$ where the sum is over all cyclic vertices. Then $\End_{\QGr kQ}(\pi^*P)$ is an ultramatricial algebra. Moreover, $\pi^*P$ is a finitely generated projective generator so the functor $\Hom_{\QGr kQ}(\pi^*P,-)$ is an equivalence of categories
$$
\QGr kQ \equiv \Mod \End_{\QGr kQ}(\pi^*P).
$$
\end{Prop}

From the equivalence $\Mod \End_{\QGr kQ}(\pi^*P) \equiv \QGr kQ$, we get an isomorphism of pre-ordered abelian groups 
$$
(\K0(\End_{\QGr kQ}(\pi^*P)),\K0^+(\End_{\QGr kQ}(\pi^*P)) \cong (\ZZ^p,\Delta(E_Q))
$$
which sends $[\End_{\QGr kQ}(\pi^*P)]$ to the element $(1,\ldots,1)$.

Suppose $kQ'$ is another path algebra of finite GK-dimension and $E_{Q'}=E_Q$. Let $P'$ be the element $\oplus e_{\mu_i}kQ'$ where the sum is over cyclic vertices of $Q'$. Then we get isomorphisms of preordered abelian groups
\begin{eqnarray*}
(\K0(\End_{\QGr kQ}(\pi^*P)),\K0(\End_{\QGr kQ}(\pi^*P)^+) \to (\ZZ^p,\Delta(E_Q)) = \\
(\ZZ^p,\Delta(E_{Q'}) \to  (\K0(\End_{\QGr kQ'}(\pi^*P')),\K0(\End_{\QGr kQ'}(\pi^*P')^+)
\end{eqnarray*}
where $[\End_{\QGr kQ}(\pi^*P)] \mapsto [\End_{\QGr kQ'}(\pi^*P')]$. 

Hence, $\End_{\QGr kQ}(\pi^*P)$ and $\End_{\QGr kQ'}(\pi^*P')$ are ultramatricial algebras for which there is an isomorphism of their Grothendieck groups which sends the order unit $[\End_{\QGr kQ}(\pi^*P)]$ to the order unit $[\End_{\QGr kQ'}(\pi^*P')]$. The following Theorem in \cite{Good} then implies $\End_{\QGr kQ}(\pi^*P)$ and $\End_{\QGr kQ'}(\pi^*P')$ are isomorphic as $k$-algebras.

\begin{Thm} [\cite{Good}, Thm $15.26$ page 221] Let $R$ and $S$ be ultrmatricial algebras. If there is an isomorphism $(\K0(R),\K0^+(R)) \to (\K0(S),\K0^+(S))$ which sends $[R] \to [S]$, then $R\cong S$ as $k$-algebras.
\end{Thm}

Hence, we now have an explicit description of the equivalence $\QGr kQ \equiv \QGr kQ'$. 
$$
\UseComputerModernTips
\xymatrix{ {\QGr kQ}\ar[r]^(0.35){\equiv} & {\Mod \End(\pi^*P)}\ar[r]^{\equiv} & {\Mod \End(\pi^*P')} & {\QGr kQ'} \ar[l]_(0.35){\equiv} }
$$

The first and third equivalences are coming from ``homing'' out of the objects $\pi^*P$ and $\pi^*P'$ respectively while the middle is due to the isomorphism 
$$
\End_{\QGr kQ}(\pi^*P)\cong \End_{\QGr kQ'}(\pi^*P').
$$

\section{Path Algebras of GK-dimension one.}
\subsection{}
This last section shows $\QGr kQ$ is a semisimple category when the path algebra has GK-dimension $1$. This also takes care of the left and right Noetherian path algebras as such algebras will be shown to have GK-dimension at most $1$.

Suppose $kQ$ is a path algebra of GK-dimension one. Given any cycle $C$ in $Q$, there are no cycles larger than $C$. Hence, if $v$ is a cyclic vertex, then there are only finitely many paths that start at $v$ which do not end at a vertex in $C$. Using the notation set up at the beginning of section \ref{sec.cyclicpoints}, the projective resolution for $P_v$ is 
$$
\UseComputerModernTips
\xymatrix{ {0}\ar[r] & {\bigoplus p^ma_1\cdots a_ibkQ}\ar[r]^(0.7){\i} & {e_vkQ}\ar[r] & {P_v}\ar[r] & {0.} }
$$
Since there are only finitely many paths of the form $bq$ when $b\neq a_i$ starts at a cyclic vertex, we know 
$$
p^ma_1\cdots a_ibkQ
$$ 
is finite dimensional and hence torsion. Hence, pushing the above exact sequence through $\pi^*$ induces an isomorphism
$$
\pi^*(e_vkQ)\cong \pi^*\cO_v.
$$
This shows in particular, that the simple objects in $\QGr kQ$ are projective.

The progenerator $\pi^*P=\oplus \pi^*e_{\nu_i}kQ$ is therefore isomorphic to the sum of the cyclic point modules
$$
\pi^*P\cong \bigoplus \pi^*\cO_{\nu_i}.
$$

\begin{lemma} \label{lemma.endoisk}
Let $M$ be a point module over an $\NN$-graded $k$-algebra $A$. Then
$$
\End_{\QGr A}(\pi^*M) \cong k.
$$
\end{lemma}
\begin{proof}
Write $M=\bigoplus ke_i$ where $e_i$ is a basis element for $M_i$. By definition,
$$
\End_{\QGr A}(\pi^*M)=\varinjlim \Hom_{\Gr A}(M',M/M'')
$$
where the direct limit is over all graded submodules $M',M''$ of $M$ such that $M/M'$ and $M''$ are torsion. Since $M$ is a point module, the only non-zero graded submodules are those of the form $M_{\geq n}$ for $n\geq 0$. Since none of these submodules are torsion and all have torsion cokernels, we get
$$
\End_{\QGr kQ}(\pi^*M)=\varinjlim_{n}\Hom_{\Gr A}(M_{\geq n},M). 
$$
If $\phi :M_{\geq n} \to M$, then for all $i\geq n$, there exists $\beta_i\in k$ such that $\phi(e_i)=\beta_ie_i$. Since $M$ is a point module, there exists $a_i\in A_1$ such that $e_{i+1}=e_ia_i$. Hence,
$$
\beta_ie_{i+1}=\beta_ie_ia_i=\phi(e_i)a_i=\phi(e_ia_i)=\phi(e_{i+1})=\beta_{i+1}e_{i+1}
$$ 
which implies $\beta_i=\beta_n$ for all $i\geq n$. This shows
$$
\Hom_{\Gr A}(M_{\geq n},M)\cong k
$$
for all $n\in \NN$. Moreover, the canonical map 
$$
\Hom_{\Gr A}(M_{\geq n},M)\to \Hom_{\Gr A}(M_{\geq n+i},M),
$$ 
which determines the direct limit, is the identity when the $\Hom$ spaces are identified with $k$. Thus,
$$
\End_{\QGr kQ}(\pi^*M)\cong k.
$$
\end{proof}

\begin{proposition}
Let $kQ$ be a path algebra of GK-dimension one. Let $P=\oplus e_ikQ$ be the object defined above. Then
$$
\End_{\QGr kQ}(\pi^*P) \cong k^n
$$
where $n$ is the number of cyclic vertices in $Q$.
\end{proposition}
\begin{proof}
By Lemma \ref{lemma.endoisk} we know $\End_{\QGr kQ}(\cO_v)=k$ for each cyclic vertex $v$. Moreover, $\cO_v \not\cong \cO_w$ if $v\neq w$. Hence,
$$
\End_{\QGr kQ}(\pi^*P)\cong \prod \End_{\QGr kQ}(\cO_v)\cong k^n
$$
where $n$ is the number of cyclic vertices.
\end{proof}

\begin{Thm} \label{thm.gk1thm}
Let $kQ$ be a path algebra of GK-dimension one. Then there is an equivalence of categories
$$
\QGr kQ \equiv \Mod k^n
$$
where $n$ is the number of cyclic vertices in $Q$.
\end{Thm}
\begin{proof}
This follows from the equivalence $\QGr kQ \equiv \Mod \End_{\QGr kQ}(\pi^*P)$ and the isomorphism $\End_{\QGr kQ}(\pi^*P)\cong k^n$ where $n$ is the number of cyclic vertices.
\end{proof}

\subsection{Noetherian path algebras.}
As before, path algebras are assumed to be have infinite dimension.
 
There is a simple condition on the quiver $Q$ which determines when $kQ$ is right or left Noetherian. 

\begin{itemize}
\item A path algebra $kQ$ is left Noetherian if and only if for each cyclic vertex $v$, there is only one arrow whose target is $v$.
\item A path algebra $kQ$ is right Noetherian if	 and only if for each cyclic vertex $v$, there is only one arrow whose source is $v$.
\end{itemize}

Notice that the presence of a doubly cyclic vertex implies the existence of a cyclic vertex which is the source of at least two distinct arrows. Hence, a left or right Notherian path algebra has no doubly cyclic vertices. Therefore, if $kQ$ is left or right Noetherian, then $kQ$ has finite GK-dimension. 

If $C_1$ and $C_2$ are distinct simple cycles of $Q$, then a path from a vertex of $C_1$ to vertex of $C_2$ implies there are cyclic vertices $v$ and $u$ such that $v$ is the source of more than one arrow and $u$ is the target of more than one arrow. Hence, there can be no paths between simple cycles in $Q$ if $kQ$ is left or right Noetherian. Therefore, if $kQ$ is left or right Noetherian, then every element of the poset $\mC(Q)$ is maximal (and minimal) which implies $kQ$ has GK-dimenion one.    

Since a right or left Noetherian path algebra has GK-dimension one, Theorem \ref{thm.gk1thm} has the following corollary.

\begin{corollary}
Let $kQ$ be a left or right Noetherian path algebra. Then
$$
\QGr kQ \equiv \Mod k^n
$$
where $n$ is the number of cyclic vertices in $Q$.
\end{corollary}
%%%%%%%%%%%%%%%%%%%%%%%%%%%%%%%%%%%%%%%%%%%%%%%%%%%%%%%%%

%%%%%%%%%%%%%%%%%%%%%%%%%%%%%%%%%%%%%%%%%%%%%%%%%%%%%%%%%

\begin{thebibliography}{12}

\bibitem{AS}
M. Artin, W.F. Schelter, Graded algebras of global dimension $3$. {\it Advances in Mathematics} 66.2 (1987): 171-216.

\bibitem{AZ}
M. Artin, JJ. Zhang, Noncommutative projective schemes. {\it Advances in Mathematics} 109.2 (1994): 228-287.

\bibitem{ATV1}
M. Artin, J. Tate, and M. Van den Bergh, Some algebras associated to automorphisms of elliptic curves. {\it The Grothendieck Festschrift} (2007): 33-85.

\bibitem{ATV2}
M.Artin, J. Tate, and M. Van den Bergh, Modules over regular algebras of dimension 3. {\it Inventiones mathematicae} 106(1) (1991): 335-388

\bibitem{Dav}
K.R. Davidson, {\it $C^*$-algebras by example.} Vol 6. American Mathematical Soc., 1996.

\bibitem{Good}
K.R. Goodearl, {\it Von Neumann regular rings.} (1979).

\bibitem{HS0}
C. Holdaway, S.P. Smith,
An equivalence of categories for graded modules over monomial algebras and path algebras of quivers,
{\it J. Algebra,}  {\bf 353} (2012) 249-260. doi:10.1016/j.jalgebra.2011.11.033.

\bibitem{HS1}
C. Holdaway, S.P. Smith, 
Corrigendum to ``An equivalence of categories for graded modules over monomial algebras and path algebras of quivers''[J. Algebra 353 (1) (2012) 249-260], Algebra (2012), doi:10.1016/j.algebra.2012.01.037

\bibitem{LM}
D Lind, B Marcus. {\it An introduction to symbolic dynamics and coding}. Cambridge University Press, 1995.

%\bibitem{Ok}
%J. Okninski, On monomial algebras. {\it Archiv der Mathematik} 50, no. 5 (1988): 417-423.

\bibitem{Sm1}
S.P. Smith,  
Category equivalences involving graded modules over path algebras of quivers. {\it Advances in Mathematics} 230.4 (2012): 1780-1810.

%\bibitem{Sm2}
%S.P. Smith,
%Shift equivalence and a category equivalence involving graded modules over path algebras of %quivers,
%arXiv:1108.4994 (2011).

\bibitem{Sm3}
S.P. Smith, ``Degenerate'' 3-dimensional Sklyanin algebras are monomial algebras. {\it Journal of Algebra} 358 (2012): 74-86.

\bibitem{U1}
V. A. Ufnarovskii, Criterion for the growth of graphs and algebras given by words, 
{\it Mat. Zametki,} {\bf 31} (1982) 465-472, Engl. Transl. 
{\it Mathematical Notes,} {\bf 31} (1982) 238-241. MR0652851 (83f:05026).

%\bibitem{U2}
%V.A. Ufnarovskii, On the use of graphs for computing a basis, growth and Hilbert series of associative %algebras. {\it Matematicheskii Sbornik} 180, no. 11 (1989): 1548-1560. Engl. Transl.
%{\it Sbornik: Mathematics} 68, no. 2 (1991): 417-428.

\bibitem{ABV}
A.B. Verevkin, On a non-commutative analogue of the category of coherent sheaves on a projective scheme, {\it Amer. Math. Soc. Transl.}, (2) {\bf 151} (1992) 

\end{thebibliography}
\end{document}